\date{\today} 
\documentclass[10pt]{article}
\usepackage{hyperref}
\usepackage{amssymb, amsmath}
\usepackage{stmaryrd}
\usepackage[final]{showlabels}
\usepackage{float}
\usepackage{abstract}

\usepackage{latexsym}

\usepackage{amsthm}

\newcommand{\g}{{\mathfrak g}}

\newcommand{\fa}{{\mathfrak a}}

\newcommand{\ff}{{\mathfrak f}}

\newcommand{\fh}{{\mathfrak h}}

\newcommand{\fk}{{\mathfrak k}}
\newcommand{\fl}{{\mathfrak l}}

\newcommand{\fq}{{\mathfrak q}}
\newcommand{\fp}{{\mathfrak p}}
\newcommand{\fr}{{\mathfrak r}}
\newcommand{\fs}{{\mathfrak s}}
\newcommand{\ft}{{\mathfrak t}}
\newcommand{\fu}{{\mathfrak u}}

\newcommand{\fz}{{\mathfrak z}}

\newcommand{\tV}{{\mathtt V}}

\renewcommand\sp{\mathfrak {sp}} 
\newcommand\hsp{\mathfrak {hsp}} 
 
\newcommand\heis{\mathfrak {heis}}

\newcommand{\1}{\mathbf{1}}

\newcommand{\cH}{\mathcal{H}}

\newcommand{\cM}{\mathcal{M}}

\newcommand{\cO}{\mathcal{O}}

\newcommand{\cW}{\mathcal{W}}

\newcommand{\N}{{\mathbb N}}

\newcommand{\R}{{\mathbb R}}
\newcommand{\C}{{\mathbb C}}

\newcommand{\K}{{\mathbb K}}

\renewcommand{\H}{{\mathbb H}}

\renewcommand{\tilde}{\widetilde}

\renewcommand{\L}{\mathop{\bf L{}}\nolimits}

\newcommand{\GL}{\mathop{{\rm GL}}\nolimits}

\newcommand{\gl}  {\mathop{{\mathfrak{gl} }}\nolimits}

\newcommand{\fsl} {\mathop{{\mathfrak{sl} }}\nolimits}

\newcommand{\hcsp}{\mathfrak{hcsp}}

\newcommand{\Fix}{\mathop{{\rm Fix}}\nolimits}

\newcommand{\ad}{\mathop{{\rm ad}}\nolimits}
\newcommand{\Ad}{\mathop{{\rm Ad}}\nolimits}

\newcommand{\fix}{\mathop{{\rm fix}}\nolimits}
\newcommand{\eff}{\mathop{{\rm eff}}\nolimits}

\newcommand{\Hom}{\mathop{{\rm Hom}}\nolimits}

\newcommand{\Pol}{\mathop{{\rm Pol}}\nolimits}

\newcommand{\Aut}{\mathop{{\rm Aut}}\nolimits}

\newcommand{\der}{\mathop{{\rm der}}\nolimits}

\newcommand{\End}{\mathop{{\rm End}}\nolimits}
\newcommand{\id}{\mathop{{\rm id}}\nolimits}

\newcommand{\rad}{\mathop{{\rm rad}}\nolimits}
\renewcommand{\dim}{\mathop{{\rm dim}}\nolimits}

\newcommand{\spec}{\mathop{{\rm spec}}\nolimits}

\newcommand{\Inn}{\mathop{{\rm Inn}}\nolimits}

\newcommand{\cone}{\mathop{{\rm cone}}\nolimits}

\newcommand{\oline}{\overline}
\newcommand{\la}{\langle}
\newcommand{\ra}{\rangle}
\newcommand{\spann}{{\rm span}}

\def\theoremname{Theorem}
\def\propositionname{Proposition}
\def\corollaryname{Corollary}
\def\lemmaname{Lemma}
\def\remarkname{Remark}
\def\conjecturename{Conjecture} 

\def\definitionname{Definition}
\def\exercisename{Exercise}
\def\examplename{Example}
\def\examplesname{Examples}
\def\problemname{Problem}
\def\problemsname{Problems}

\def\@thmcounter#1{\noexpand\arabic{#1}}
\def\@thmcountersep{}
\def\@begintheorem#1#2{\it \trivlist \item[\hskip 
\labelsep{\bf #1\ #2.\quad}]}
\def\@opargbegintheorem#1#2#3{\it \trivlist
      \item[\hskip \labelsep{\bf #1\ #2.\quad{\rm #3}}]}
\makeatother
\newtheorem{theor}{\theoremname}[section]
\newtheorem{propo}[theor]{\propositionname}
\newtheorem{coro}[theor]{\corollaryname}
\newtheorem{lemm}[theor]{\lemmaname}

\newenvironment{thm}{\begin{theor}\it}{\end{theor}}

\newenvironment{prop}{\begin{propo}\it}{\end{propo}}

\newenvironment{cor}{\begin{coro}\it}{\end{coro}}

\newenvironment{lem}{\begin{lemm}\it}{\end{lemm}}

\newtheorem{rema}[theor]{\remarkname}

\newenvironment{rem}{\begin{rema}\rm}{\end{rema}}

\newtheorem{stepnow}[theor]{}

\newtheorem{defin}[theor]{\definitionname} %% write

\newenvironment{definition}{\begin{defin}\rm}{\end{defin}}

\newtheorem{exerc}{\exercisename}[section]

\newtheorem{exa}[theor]{\examplename}

\newenvironment{example}{\begin{exa}\rm}{\end{exa}}

\newtheorem{exas}[theor]{\examplesname}

\newtheorem{conj}[theor]{\conjecturename}

\newtheorem{pro}[theor]{\problemname}

\newtheorem{prs}[theor]{\problemsname}

\newcommand{\pmat}[1]{\begin{pmatrix} #1 \end{pmatrix}}

\usepackage{bbold}

\newcommand{\Cmin}{C_{\mathrm{min}}}
\newcommand{\Cmax}{C_{\mathrm{max}}}

\newcommand{\bbone}{\mathbb{1}}

\newcommand{\p}{\partial}

\newcommand{\tran}{\mathsf{T}}

\usepackage{enumitem}
\setlist[enumerate,1]{label={\rm (\alph*)}}
\setlist[enumerate,2]{label={\rm (\roman*)}}

\makeatletter
\def\blfootnote{\xdef\@thefnmark{}\@footnotetext}
\makeatother

\begin{document}

\title{Lie wedges of endomorphism semigroups of standard subspaces in admissible Lie algebras} 

\author{Daniel Oeh\footnote{Department Mathematik, Friedrich--Alexander--Universit\"at Erlangen--N\"urnberg, Cauerstr. 11, D-91058 Erlangen, Germany, oehd@math.fau.de} \,\footnote{Supported by DFG-grant Ne 413/10-1}} 

\maketitle

\blfootnote{2010 \textit{Mathematics Subject Classification.} Primary 22E45; Secondary 81R05, 81T05}
\blfootnote{\textit{Key words and phrases.} Lie algebra, invariant cone, symplectic module of convex type, standard subspace.}

\begin{abstract}
  Let \(\g\) be a real finite-dimensional Lie algebra containing pointed generating invariant closed convex cones. We determine those derivations \(D\) of \(\g\) which induce a 3-grading of the form \(\g = \g_{-1} \oplus \g_0 \oplus \g_1\) on \(\g\) such that the \((\pm 1)\)-eigenspaces \(\g_{\pm 1}\) of \(D\) are generated by the intersections with generating cones of the form \(W = \cO_f^*\), where \(\cO_f\) is the coadjoint orbit of a linear functional \(f \in \fz(\g)^*\) and \(\cO_f^*\) is the dual cone of \(\cO_f\). In particular, we show that, if \(\g\) is solvable, no such derivation except the trivial one exists.

  This continues our classification of Lie algebras generated by Lie wedges of endomorphism semigroups of standard subspaces. The classification is motivated by the relation of nets of standard subspaces to Haag--Kastler nets of von Neumann algebras in Algebraic Quantum Field Theory.
\end{abstract}

\section{Introduction}
\label{sec:intro}

In this paper, we continue our study of the following classification problem: Let \((\g,\tau)\) be a symmetric finite-dimensional real Lie algebra and let \(W \subset \g\) be an \((e^{\ad \g})\)-invariant closed convex cone.
Let \(\fh\) be the \(1\)-eigenspace of \(\tau\) and \(\fq\) be the \((-1)\)-eigenspace of \(\tau\).
Moreover, let \(h \in \fh\).
Our goal is to classify the Lie algebras \(\g(W,\tau,h)\) which are generated by the \emph{Lie wedges}
\[C_-(W,\tau,h) \oplus [C_-(W,\tau,h), C_+(W,\tau,h)] \oplus C_+(W,\tau,h), \quad \text{where } C_\pm(W,\tau,h) := \g_{\pm 1}(h) \cap \fq \cap (\pm W),\]
and \(\g_\lambda(h)\) denotes the \(\lambda\)-eigenspace of \(\ad(h)\) for \(\lambda \in \R\).

Our motivation for studying the Lie algebras \(\g(W,\tau,h)\) comes from Algebraic Quantum Field Theory and, more specifically, standard subspaces: For a complex Hilbert space \(\cH\), a \emph{standard subspace} is a closed real subspace \(\tV\) of \(\cH\) such that \(\tV \cap i\tV = \{0\}\) and \(\tV + i\tV\) is dense in \(\cH\).
Every von Neumann algebra \(\cM\) in the algebra of bounded operators on \(\cH\) with a cyclic and separating vector \(\Omega \in \cH\) (that is, \(\cM\Omega\) is a dense subspace of \(\cH\) and the map \(\cM \ni A \mapsto A\Omega\) is injective) yields a standard subspace
\[\tV_\cM := \oline{\{A \Omega : A \in \cM, A^* = A\}}.\]
In particular, given a Haag--Kastler net (cf.\ \cite{Ha96}), i.e.\ a net of von Neumann algebras indexed by open regions in some spacetime, we obtain a corresponding net of standard subspaces if we assume the existence of a common cyclic and separating vector. Conversely, nets of von Neumann algebras can be constructed from nets of standard subspaces using Second Quantization (cf.\ \cite[Sec.\ 6]{NO17} and \cite{Ar63}).

In this context, we are particularly interested in inclusions of von Neumann algebras and thus, by the argument above, in inclusions of standard subspaces. More concretely, let \((U,\cH)\) be a strongly continuous unitary representation of a Lie group \(G\), let \((\cM,\Omega)\) be a von Neumann algebra with a cyclic and separating vector and let \(\tV := \tV_\cM\). Then the inclusion order on the orbit \(G.\cM := \{U(g)\cM U(g)^{-1} : g \in G\}\), respectively on the orbit \(U(G)\tV\), is determined by the subsemigroups
\[S_\cM := \{g \in G : U(g)\cM U(g)^{-1} \subset \cM\} \quad \text{and} \quad S_\tV := \{g \in G : U(g)\tV \subset \tV\}.\]
While we always have \(S_\cM \subset S_\tV\), this inclusion is proper in general, as shown in \cite[Ex.\ 5.1]{Ne19}.
Since standard subspaces are easier to study than the von Neumann algebras they were constructed from (we refer to \cite{Lo08} for more information about their structure), our main focus is the semigroup \(S_\tV\), which is also called the \emph{endomorphism semigroup of \(\tV\)}.
We make the following assumption on \(\tV\): Let \(\tau_G \in \Aut(G)\) be an involutive automorphism and let \(\tau := \L(\tau_G)\) and \(\g := \L(G)\), where \(\L\) denotes the Lie functor.
Suppose that the unitary representation \(U\) can be extended to a representation of the semidirect product \(G_\tau := G \rtimes \{\1, \tau_G\}\) in such a way that \(J := U(\tau_G)\) is antiunitary.
Let \(h \in \fh\) and let \(\p U(h)\) be the infinitesimal generator of \(t \mapsto U(\exp(th))\).
For \(\Delta := e^{2\pi i \p U(h)}\), we then obtain a standard subspace \(\tV_{\tau, h} := \Fix(J\Delta^{1/2})\) (cf.\ \cite{NO17}).
Our assumption is that \(\tV = \tV_{\tau,h}\) for some involutive automorphism \(\tau_G \in \Aut(G)\) and \(h \in \fh\).

Since \(S_\tV\) is a subsemigroup of the Lie group \(G\), we can consider its \emph{Lie wedge}
\[\L(S_\tV) := \{x \in \g : \exp(\R_{\geq 0} x) \subset S_\tV\}.\]
In the underlying setting, we can describe \(\L(S_\tV)\) in terms of the Lie algebraic data that we have used to construct \(\tV\): Let \(C_U := \{x \in \g : -i\p U(x) \geq 0\}\) be the \emph{positive cone} of \(U\) and suppose that the kernel of \(U\) is discrete, so that \(C_U\) is pointed, i.e.\ \(C_U \cap (-C_U) = \{0\}\). Then, by \cite[Thm.\ 4.4]{Ne19}, we have
\begin{equation}
  \label{eq:std-subspace-liewedge}
  \L(S_\tV) = C_-(C_U, \tau, h) + \fh_0(h) + C_+(C_U, \tau, h)
\end{equation}
and the cones \(C_{\pm}(C_U, \tau, h)\) span abelian subspaces. In particular, the Lie algebra generated by \(C_-(C_U, \tau, h)\) and \(C_+(C_U, \tau, h)\) is \(\g(C_U, \tau, h)\).

In \cite{Oeh20}, we showed that, if \(\g\) is hermitian simple, then the Lie subalgebras \(\g(W,\tau,h)\) are hermitian simple, and we classified  all possible Lie subalgebras \(\g(W,\tau,h)\) of \(\g\) up to isomorphy for a fixed involutive automorphism \(\tau\).
To address the general case, we first want to classify all possible Lie subalgebras \(\g(W,\tau,h)\) for a fixed \(\g\) up to isomorphy.
Since the restriction of \(\tau\) to \(\g(W,\tau,h)\) coincides with \(e^{i\pi\ad h}\lvert_{\g(W,\tau,h)}\) (cf.\ \cite[Thm.\ 4.4]{Ne19}), we may without loss of generality assume that \(\ad h\) is diagonalizable on \(\g\) with eigenvalues \(0\) and \(\pm 1\) and that \(\tau = e^{i\pi \ad h}\lvert_\g\).
A crucial step towards a classification in the non-reductive case is to study the Lie algebras of the form \(\g_e = \g \rtimes \R h\), where \(\g\) is a Lie algebra containing a generating invariant closed convex cone and \(h\) acts on \(\g\) by a derivation which induces a 3-grading on \(\g\) of the form \(\g = \g_{-1}(h) \oplus \g_0(h) \oplus \g_1(h)\).

In the hermitian setting of \cite{Oeh20}, it turns out that the subalgebras \(\g(W,\tau,h)\) do not depend on the choice of the cone as long as it is proper, which is due to their relation to simple euclidean Jordan algebras.
For general Lie algebras, it is not clear whether this simplification can be made.
However, in view of \eqref{eq:std-subspace-liewedge}, we only need to consider those pointed invariant cones \(W\) in our classification which occur as the positive cones of unitary representations of Lie groups \(G\) with Lie algebra \(\g\).
Since we also expect that \eqref{eq:std-subspace-liewedge} adapts to decompositions of the representation \(U\) into irreducible subrepresentations, we may assume that \(U\) is irreducible.
Under these assumptions, it is sufficient to consider the generating invariant cones \(W \subset \g\) which are of the form \(W = W_f := \cO_f^*\) for some \(f \in \fz(\g)^*\), where \(\cO_f\) is the \emph{coadjoint orbit} of the extension of \(f\) to \(\g\) along a decomposition \(\g = \ft + [\ft, \g]\) into a compactly embedded Cartan subalgebra \(\ft \subset \g\) and its image under the adjoint representation.
Moreover, \(\cO_f^*\) denotes the dual cone of \(\cO_f\). 
We will explain this reduction in more detail in Remark \ref{rem:spindler-cones-unitary-rep}.

In this paper, we give a description of all derivations \(D\) of \(\g\) inducing 3-gradings on \(\g\) which are suitable for the study of the Lie subalgebras \(\g(W,\tau,h)\). We then show that, for the generating invariant cones \(W_f \subset \g\) as above, the intersection \(\g_{\pm 1}(D) \cap W_f\) of \(W_f\) with the \((\pm 1)\)-eigenspaces of \(D\) is generating.

This paper consists of two main sections: In Section \ref{sec:liealg-inv-cone}, we consider the structure theory of Lie algebras \(\g\) containing invariant convex cones.
The main tool we will be using to study these Lie algebras is Spindler's construction from \cite{Sp88}, which we introduce in Section \ref{sec:lie-adm}.
It allows us to view \(\g\) as a semidirect product \(\fu \rtimes \fl\), where \(\fl\) is reductive and \(\fu\) is the maximal nilpotent ideal of \nolinebreak\(\g\).
Moreover, \(\fu\) can be written as \(\fu = \fz(\g) + V\), where \(V\) is a symplectic \(\fl\)-module of convex type. In Section \ref{sec:adm-reducompl}, we show that this decomposition can be chosen in such a way that \(\fl\) is \(D\)-invariant with \(D(\fz(\fl)) = \{0\}\) (cf.\ Theorem \ref{thm:adm-deriv-3grad-reducompl}). The proof uses results on algebraic groups, which we recall in Appendix \ref{sec:app-alg-grp}.

We then address our classification problem in Section \ref{sec:liewedge-adm}.
Our first main result is that, if \(\g\) is solvable with \(\fz(\g) \subset [\g,\g]\) and \(W \subset \g\) is a generating invariant closed convex cone such that \(W \cap -W\) is central in \(\g\), then there exists no non-trivial derivation \(D\) of \(\g\) which induces a 3-grading on \(\g\) such that the intersections \(W \cap \g_{\pm 1}(D)\) generate \(\g_{\pm 1}(D)\).
The general case is studied in Section \ref{sec:deriv-adm-gen-case}:
We first give a description of all derivations \(D\) of \(\g\) inducing a 3-grading using a suitable decomposition of \(\g\) in terms of Spindler's construction (cf. Classification Theorem \ref{thm:adm-deriv-3grad}).
In the Non-degeneration Theorem \ref{thm:adm-3grad-intcones}, we show that the intersections \(W_f \cap \g_{\pm 1}(D)\) are generating for these derivations if the cone \(W_f\) is generating for \(f \in \fz(\g)^*\). Finally, we discuss various examples of derivations that induce 3-gradings.

\subsection*{Notation}

\begin{itemize}
  \item For a vector space \(V\) over \(\K \in \{\R,\C\}\), a linear endomorphism \(A\) on \(V\), and \(\lambda \in \K\), we denote the \(\lambda\)-eigenspace of \(A\) by \(V_\lambda(A)\).

  \item Let \(\g\) be a Lie algebra and let \(x \in \g\). For an \(\ad(x)\)-invariant subspace \(V \subset \g\), we define \(V_\lambda(x) := \ker(\ad(x) - \lambda\id_\g)\).
  \item We denote by \(\Inn(\g) := \la e^{\ad \g} \ra\) the group of inner automorphisms of \(\g\).
  \item For a Lie algebra \(\g\), we denote the set of derivations of \(\g\) by \(\der(\g)\).
  \item For a subset \(W \subset V\) in a real vector space \(V\), we denote the \emph{dual cone of \(W\)} by
  \[W^* := \{f \in V^* : f(W) \subset \R_{\geq 0}\}.\]
  \item Let \(A\) be an algebra and let \(\tau\) be an involutive automorphism or antiautomorphism of \(A\). Then we denote the \((\pm 1)\)-eigenspace of \(\tau\) in \(A\) by \(A^{\pm \tau}\).
  \item Let \(V\) be a real vector space and let \(V_\C\) be its complexification. For \(z = x + iy\), where \(x,y \in V\), we define \(\oline{z} := x - iy\) and \(z^* := -x + iy\).
\end{itemize}

\section{Lie algebras containing invariant convex cones}
\label{sec:liealg-inv-cone}

In this section, we recall the structure theory of Lie algebras containing invariant convex cones.
We introduce Spindler's construction and recall some results about symplectic modules of convex type.
All Lie algebras containing pointed generating invariant closed convex cones are of the form \(\fu \rtimes \fl\), where \(\fu\) is nilpotent and \(\fl\) is a reductive Lie algebra for which all simple non-compact ideals are hermitian.
In Section \ref{sec:adm-reducompl}, we discuss the reductive complements \(\fl\) of such Lie algebras which are invariant under a given semisimple derivation.

\subsection{Admissible Lie algebras}
\label{sec:conv-liealg}

Lie algebras containing pointed generating invariant closed convex cones are called \emph{admissible}. Such Lie algebras contain compactly embedded Cartan subalgebras, which yield a root space decomposition. We recall some important facts about admissible Lie algebras in this section which will be used later. The content of this section is mostly identical to \cite[Section 2.1]{Oeh20}.

\begin{definition}
  A real finite dimensional Lie algebra \(\g\) is called \emph{admissible} if it contains a generating invariant closed convex subset \(C\) with \(H(C) := \{x \in \g : C + x = C\} = \{0\}\), i.e.\ \(C\) contains no affine lines.
\end{definition}

\begin{definition}
  Let \(\g\) be a Lie algebra and \(\fk \subset \g\) be a subalgebra. Then \(\fk\) is said to be \emph{compactly embedded} if the subgroup generated by \(e^{\ad \fk}\) is relatively compact in \(\Aut(\g)\).
\end{definition}

Every admissible Lie algebra \(\g\) contains a compactly embedded Cartan subalgebra \(\ft\) by \cite[Thm.\ VII.3.10]{Ne00}.
Let \(\fk\) be a maximal compactly embedded subalgebra which contains \(\ft\).
Since \(\ft_\C\) is a Cartan subalgebra of the complexification \(\g_\C\) of \(\g\), we obtain a root space decomposition of \nolinebreak\(\g_\C\).
We denote the corresponding set of roots by \(\Delta := \Delta(\g_\C, \ft_\C)\).
The set of roots can be further decomposed into the following subsets:

\begin{itemize}
  \item The set of \emph{compact roots} \(\Delta_k := \{\alpha \in \Delta : \g_\C^\alpha \subset \fk_\C\}\).
    The set of \emph{non-compact roots} is denoted by \(\Delta_p := \Delta \setminus \Delta_k\).
  \item The set of \emph{semisimple roots} \(\Delta_s := \{\alpha \in \Delta : (\exists Z \in \g_\C^\alpha)\, \alpha([Z,Z^*]) \neq 0\}\). The roots in the complement \(\Delta_r := \Delta \setminus \Delta_s\) are called \emph{solvable}.
\end{itemize}

Let \(\Delta^+ \subset \Delta\) be a system of positive roots. Then there exists an element \(X \in i\ft\) such that \(\Delta^+ = \{\alpha \in \Delta : \alpha(X) > 0\}\) and \(\alpha(X) \neq 0\) for all \(\alpha \in \Delta\). We say that \(\Delta^+\) is \emph{adapted} if \(\beta(X) > \alpha(X)\) for all \(\alpha \in \Delta_k, \beta \in \Delta_p^+\).

A Lie algebra \(\g\) which contains a compactly embedded Cartan subalgebra is called \emph{quasihermitian} if there exists an adapted system of positive roots.
An equivalent condition is that \(\fz_\g(\fz(\fk)) = \fk\) for a maximal compactly embedded subalgebra \(\fk\) of \(\g\) (cf.\ \cite[Prop.\ VII.2.14]{Ne00}). Every simple quasihermitian Lie algebra is either compact or satisfies \(\fz(\fk) \neq \{0\}\). A simple quasihermitian non-compact Lie algebra is called \emph{hermitian}.

\begin{definition}
  \label{def:weyl-group-compact}
  Let \(\g\) be admissible, let \(\fk\) be a compactly embedded subalgebra of \(\g\), and let \(\ft \subset \fk\) be a compactly embedded Cartan subalgebra. Denote by \(\Delta := \Delta(\g_\C,\ft_\C)\) the corresponding set of roots. For every \(\alpha \in \Delta_k\), there exists a unique element \(\alpha^\vee \in [\g_\C^\alpha,\g_\C^{-\alpha}]\) with \(\alpha(\alpha^\vee) = 2\). We call the group \(\cW_\fk\) generated by the reflections
  \[s_\alpha : \ft \rightarrow \ft, \quad X \mapsto X - \alpha(X)\alpha^\vee, \quad \alpha \in \Delta_k,\]
  the \emph{Weyl group of the pair \((\fk, \ft)\)} (cf.\ \cite[Def.\ VII.2.8]{Ne00}).
\end{definition}

With \(\g,\fk,\ft\), and \(\Delta\) as in Definition \ref{def:weyl-group-compact} and a system of positive roots \(\Delta^+ \subset \Delta\), we define the convex cones
\[\Cmin := \Cmin(\Delta_p^+) := \cone\{i[Z_\alpha,\oline{Z_\alpha}] : Z_\alpha \in \g_\C^\alpha, \alpha \in \Delta_p^+\} \subset \ft\]
and
\[\Cmax := \Cmax(\Delta_p^+) := \{X \in \ft : (\forall \alpha \in \Delta_p^+)\, i\alpha(X) \geq 0\}.\]
Then, for every pointed generating closed convex cone \(W \subset \g\), there exists a unique adapted system of positive roots such that \(\Cmin \subset W \cap \ft \subset \Cmax\) (cf.\ \cite[Thm.\ VII.3.8]{Ne00}).

Conversely, every generating \(\cW_\fk\)-invariant cone \(C \subset \ft\) with \(\Cmin \subset C \subset \Cmax\) for some adapted system of positive roots uniquely determines a generating invariant closed convex cone \(W\) with \(W \cap \ft = C\) (cf.\ \cite[Thm.\ VIII.3.21]{Ne00}).

\subsection{Hermitian Lie algebras}
\label{sec:herm-liealg}

In this section, we briefly recall the most important facts about hermitian simple Lie algebras which will be used in the following sections.

The list of isomorphy classes of hermitian simple Lie algebras (cf.\ for example \cite[Thm.\ A.V.1]{Ne00}) can be divided into two types: A hermitian Lie algebra \(\g\) is said to be of \emph{tube type} if there exists an element \(h \in \g \setminus \{0\}\) such that \(\ad h\) induces a 3-grading on \(\g\), i.e.\ we have \(\g = \g_{-1}(h) \oplus \g_0(h) \oplus \g_1(h)\). If no such element exists, then \(\g\) is said to be of \emph{non-tube type}. Let \(\fa \subset \fp\) be a maximal abelian subspace of \(\g = \fk \oplus \fp\) with respect to a fixed Cartan decomposition. By Moore's Theorem (cf.\ \cite[Thm.\ A.4.4]{HO97}), the restricted root system \(\Sigma \subset \fa^*\) is of type (\(C_r\)) if \(\g\) is of tube type, i.e.\ we have
\[\Sigma = \{\pm(\varepsilon_j \pm \varepsilon_i) : 1 \leq i < j \leq r\} \cup \{\pm 2\varepsilon_j : 1 \leq j \leq r\} \cong C_r.\]
An inspection of \(\Sigma\) shows that, for a fixed Cartan decomposition \(\g = \fk \oplus \fp\) and a fixed maximal abelian subspace \(\fa \subset \fp\), there exists only one Weyl group orbit of elements \(h\) inducing a 3-grading on \(\g\) of the form \(\g = \g_{-1}(h) \oplus \g_0(h) \oplus \g_1(h)\) (cf.\ \cite[Thm.\ 3.10]{MN20}). In particular, those elements \(h \in \g \setminus \{0\}\) for which \(\ad h\) induces a 3-grading on \(\g\) are conjugate under inner automorphisms.

\begin{example}
  \label{ex:herm-sp}
  Consider the hermitian simple Lie algebra \(\g = \sp(2n,\R)\), where \(n \in \N\). We fix the Cartan involution \(\theta(X) := -X^\tran\) and denote the corresponding Cartan decomposition by \(\g = \fk \oplus \fp\). Then a maximal abelian subspace \(\fa \subset \fp\) is given by the diagonal matrices in \(\g\). The element
  \[h := \frac{1}{2} \pmat{\bbone_n & 0 \\ 0 & -\bbone_n} \in \fp\]
  induces a 3-grading on \(\g\) via the adjoint representation. In particular, for all elements \(h \in \g \setminus \{0\}\) such that \(\ad h\) induces a 3-grading on \(\g\), the linear map \(2h\) is an antisymplectic involution on \(\R^{2n}\).

  For \(\fsl(2,\R) = \sp(2,\R)\), we have
  \[h = \tfrac{1}{2}H, \quad \text{where} \quad H := \pmat{1 & 0 \\ 0 & -1}.\]
\end{example}

For every hermitian simple Lie algebra \(\g\), the center of a maximal compactly embedded subalgebra \(\fk \subset \g\) is one-dimensional. In particular, there exists an element \(H_0 \in \fz(\fk)\) such that \(\spec(\ad H_0) = \{0,i,-i\}\). Such an element is called an \emph{H-element}:

\begin{definition}
  \label{def:h-element}
  (cf.\ \cite[Def.\ II.1]{HNO94}) Let \(\g\) be a reductive Lie algebra. Then \(H_0 \in \g\) is called an \emph{\(H\)-element} if \(\ker \ad(H_0)\) is a maximal compactly embedded subalgebra of \(\g\) and \(\spec(\ad H_0) = \{0,i,-i\}\). The pair \((\g, H_0)\) is called a \emph{Lie algebra of hermitian type}.
\end{definition}

Every \(H\)-element \(H_0\) of a Lie algebra \(\g\) of hermitian type determines a Cartan decomposition by \(\fk := \ker(\ad H_0)\) and \(\fp := [H_0,\g]\).

In the case of the hermitian simple Lie algebra \(\fsl(2,\R)\), an \(H\)-element is given by \(\frac{1}{2}U\), where
  \[U:= \pmat{0 & 1 \\ -1 & 0},\]
  and we have \(\fk = \ker(\ad U) = \R U.\)

  \begin{definition}
    A homomorphism \(\kappa : \g \rightarrow \tilde \g\) between two Lie algebras \((\g,H_0)\) and \((\tilde \g, \tilde H_0)\) of hermitian type is called an \emph{\((H_1)\)-homomorphism} if \(\kappa \circ \ad H_0 = \ad \tilde H_0 \circ \kappa\) and is called an \emph{\((H_2)\)-homomorphism} if \(\kappa(H_0) = \tilde H_0\). In either case, we denote this compatibility by \(\kappa: (\g,H_0) \rightarrow (\tilde \g, \tilde H_0)\).
  \end{definition}

\begin{lem}
  \label{lem:herm-lie-3grad-h2-homo}
  Let \((\g,H_0)\) be a hermitian simple Lie algebra of tube type with \(H\)-element \(H_0\) and let \(h \in [H_0, \g] \setminus \{0\}\) such that \(\g = \g_{-1}(h) \oplus \g_0(h) \oplus \g_1(h)\). Then there exists an \(H_2\)-homomorphism \(\kappa: (\fsl(2,\R), \frac{1}{2}U) \rightarrow (\g,H_0)\) such that \(\kappa(\frac{1}{2}H) = h\) (cf.\ {\rm Example \ref{ex:herm-sp}}).
\end{lem}
\begin{proof}
  As we have already noted at the beginning of Section \ref{sec:herm-liealg}, all elements \(h \in \g\) with \(\spec(\ad h) = \{0, 1, -1\}\) that induce 3-gradings on \(\g\) with via the adjoint representation are conjugate under inner automorphisms, so that it suffices to prove the claim for one such element. Denote the Cartan decomposition determined by \(H_0\) by \(\g = \fk \oplus \fp\) and fix a maximal abelian subspace \(\fa \subset \fp\). Then the arguments in \cite[Rem.\ 2.14(b)]{Oeh20} show that there exists an element \(h \in \fa\) such that \(\g = \g_{-1}(h) \oplus \g_0(h) \oplus \g_1(h)\) and an \((H_2)\)-homomorphism \(\kappa: (\fsl(2,\R), \frac{1}{2}U) \rightarrow (\g, H_0)\) with \(\kappa(\frac{1}{2}H) = h\), which finishes the proof.
\end{proof}

\subsection{Spindler's construction}
\label{sec:lie-adm}

After studying real simple Lie algebras containing invariant non-trivial cones, we now focus on more general finite-dimensional Lie algebras with this property.
For a reductive admissible Lie algebra, every simple ideal is either hermitian or compact (cf.\ e.g.\ \cite[Lem.\ 2.8]{Oeh20}).

Non-reductive admissible Lie algebras can be constructed as follows: Let
\begin{enumerate}
  \item \(\fl\) be a reductive Lie algebra,
  \item \(V\) be an \(\fl\)-module,
  \item \(\fz\) be a vector space, and
  \item \(\beta: V \times V \rightarrow \fz\) be a skew-symmetric bilinear map which is \emph{\(\fl\)-invariant} in the sense that
    \begin{equation}
      \label{eq:sympl-mod-spindler}
      \beta(x.v,w) + \beta(v,x.w) = 0 \quad \text{for all } x \in \fl, v,w \in V.
    \end{equation}
\end{enumerate}
Then \(\g(\fl,V,\fz,\beta) := V \times \fz \times \fl\) with the bracket
\begin{equation*}
  [(v,z,x),(v',z',x')] := (x.v' - x'.v, \beta(v,v'), [x,x'])
\end{equation*}
is a Lie algebra. This is also known as \emph{Spindler's construction} (cf.\ \cite{Sp88}).

\begin{thm}{\rm (Spindler's Construction)}
  \label{thm:spindlers-constr}
  \begin{enumerate}
    \item Let \(\g\) be an admissible Lie algebra. Then the maximal nilpotent ideal \(\fu \subset \g\) satisfies \([\fu,\fu] \subset \fz(\g)\). Let \(\ft \subset \g\) be a compactly embedded Cartan subalgebra. Then there exists a reductive \(\ft\)-invariant subalgebra \(\fl \subset \g\) such that, for \(V := [\ft,\fu]\), the bilinear map
      \[\beta: V \times V \rightarrow \fz(\g), \quad (x,y) \mapsto [x,y],\]
      satisfies \eqref{eq:sympl-mod-spindler} with respect to the action \(x.v := [x,v]\) for \(x \in \fl\) and \(v \in V\). Moreover, \(\ft_\fl := \ft \cap \fl\) acts on \(V\) by semisimple endomorphisms with purely imaginary spectrum such that
      \begin{equation}
        \label{eq:spindler-center-V}
        \fz_V(\ft_\fl) = \{v \in V : \ft_\fl.v = \{0\}\} = \{0\}.
      \end{equation}
      Moreover, we have \(\g \cong \g(\fl,V,\fz(\g),\beta)\).
    \item Conversely, let \(\g = \g(\fl,V,\fz,\beta)\) for some reductive Lie algebra \(\fl\), an \(\fl\)-module \(V\) with \eqref{eq:sympl-mod-spindler}, and a vector space \(\fz\). Let \(\ft_\fl \subset \fl\) be a compactly embedded Cartan subalgebra of \(\fl\). If \(\fz_V(\ft_\fl) = \{0\}\), then the following holds:
      \begin{enumerate}
        \item \(\fz(\g) = \{0\} \times \fz \times \fz_{\fz(\fl)}(V)\).
        \item \(\fu := V + \fz(\g)\) is the maximal nilpotent ideal of \(\g\) and \([\fu,\fu] \subset \fz(\g)\).
        \item \(\ft := \{0\} \times \fz \times \ft_\fl\) is a compactly embedded Cartan subalgebra of \(\g\).
        \item \(\g \cong \fu \rtimes \fl\).
        \item \(\fr := V \times \fz \times \fz(\fl)\) is the radical of \(\g\).
      \end{enumerate}
  \end{enumerate}
\end{thm}
\begin{proof}
  Let \(\g\) be an admissible Lie algebra. Then it follows from \cite[Thm.\ VII.3.10]{Ne00} and \cite[Prop.\ VII.2.23]{Ne00} that the maximal nilpotent ideal \(\fu\) of \(\g\) satisfies \([\fu,\fu] \subset \fz(\g)\). The remaining assertions follow from \cite[Thm.\ VII.2.26]{Ne00}.
\end{proof}

\begin{definition}
  \label{def:symp-mod}
  Let \(\g\) be a Lie algebra.
  \begin{enumerate}
    \item Let \(V\) be a \(\g\)-module and let \(\Omega: V \times V \rightarrow \R\) be a symplectic form on \(V\) such that
      \[\Omega(x.v,w) + \Omega(v,x.w) = 0 \quad \text{for all } x \in \g, v,w \in V.\]
      Then \((V,\Omega)\) is called a \emph{symplectic \(\g\)-module}.
    \item Two symplectic \(\g\)-modules \((V,\Omega)\) and \((W,\Omega')\) are said to be \emph{equivalent} if there exists a \(\g\)-module isomorphism \(\varphi : V \rightarrow W\) such that
      \[\Omega(v,w) = \Omega'(\varphi(v),\varphi(w)) \quad \text{for all } v,w \in V.\]
    \item A symplectic \(\g\)-module \((V,\Omega)\) is called a \emph{symplectic \(\g\)-module of convex type} if there exists an element \(x \in \g\) such that the Hamiltonian function
      \[H_x : V \rightarrow \R, \quad H_x(v) := \Omega(x.v,v),\]
      is positive definite.
  \end{enumerate}
\end{definition}

\begin{rem}
  \label{rem:sympmod-conv-defs}
  Note that the term \emph{symplectic \(\g\)-module of convex type} is sometimes defined in a more general manner than in Definition \ref{def:symp-mod} (cf.\ for example \cite{Ne94a},\cite{Neu00}): In these cases, a symplectic \(\g\)-module \((V,\Omega)\) is said to be of convex type if the convex cone
  \begin{equation}
    \label{eq:sympmod-cone}
    W_V := \{x \in \g : (\forall v \in V)\, \Omega(x.v,v) \geq 0\}
  \end{equation}
  is generating in \(\g\). If this is the case, then \(V = V_{\fix} \oplus V_{\eff}\) is an orthogonal direct sum of symplectic \(\g\)-modules, where
  \[V_{\fix} := \{v \in V : \g.v = \{0\}\} \quad \text{and} \quad V_{\eff} := \spann (\g.V),\]
  and \(V_{\eff}\) is a symplectic \(\g\)-module of convex type in the sense of Definition \ref{def:symp-mod} (cf.\ \cite[Prop.\ II.5]{Ne94a}, \cite[Prop.\ II.23]{Ne94a}). Thus, the two definitions are equivalent if \(V_{\fix} = \{0\}\).

  As we will see in Theorem \ref{thm:spindler-adm}, the submodule \(V_{\fix}\) is trivial for all symplectic \(\g\)-modules \((V,\Omega)\) that we consider here.
\end{rem}

The following theorem characterizes those Lie algebras arising from Spindler's Construction which are admissible:

\begin{thm}
  \label{thm:spindler-adm}
  Let \(\fl\) be a reductive Lie algebra, let \(\ft_\fl \subset \fl\) be a compactly embedded Cartan subalgebra, and \(V\) be an \(\fl\)-module with \(\fz_V(\ft_\fl) = \{0\}\) and \(\fz_{\fz(\fl)}(V) = \{0\}\). If \(\fz\) is a vector space and \(\beta: V \times V \rightarrow \fz\) is an \(\fl\)-invariant skew-symmetric map, then the Lie algebra \(\g(\fl, V, \fz, \beta)\) is admissible if and only if
  \begin{enumerate}
    \item \(\fl\) is quasihermitian and
    \item there exists an element \(f \in \fz^*\) such that \((V,f \circ \beta)\) is a symplectic \(\fl\)-module of convex type. In this case, \(f\) is contained in the interior of the generating convex cone
      \[C_{\mathrm{min},\fz}^* := (C_{\mathrm{min}} \cap \fz)^* \subset \fz^*\]
      for some adapted positive root system. 
  \end{enumerate}
\end{thm}
\begin{proof}
  cf.\ \cite[Thm.\ VIII.2.7]{Ne00} and \cite[Prop.\ VIII.2.9(i)]{Ne00}.
\end{proof}

\begin{rem}
  \label{rem:spindler-adm-conv}
  Let \(\g\) be an admissible Lie algebra and let \(\g = \g(\fl,V,\fz(\g),\beta)\) be a description of \(\g\) in terms of Spindler's construction as in Theorem \ref{thm:spindlers-constr}. Then the conditions \(\fz_V(\ft_\fl) = \{0\}\) and \(\fz_{\fz(\fl)}(V) = \{0\}\) of Theorem \ref{thm:spindler-adm} are satisfied.
\end{rem}

\begin{rem}
  \label{rem:spindler-adm-levi-comm}
  Let \(\g = \g(\fl,V,\fz,\beta)\) be an admissible Lie algebra with \(V \neq \{0\}\).
  Denote by \(\fs := [\fl,\fl]\) the semisimple part of the reductive Lie algebra \(\fl\).
  Since the radical of \(\g\) is given by \(\fr = V \times \fz \times \fz(\fl)\) (cf.\ Theorem \ref{thm:spindlers-constr}), the subalgebra \(\fs\) is a Levi complement in \(\g\).
  Moreover, since \(V\) is of convex type, we also have \(V = \spann(\fl.V)\) (cf.\ Remark \ref{rem:sympmod-conv-defs}). Hence, the commutator algebra of \(\g\) is given by \([\g,\g] = V \times [V,V] \times \fs = \g(\fs,V,[V,V],\beta)\).
\end{rem}

We recall a few facts about symplectic Lie algebra modules of convex type:

\begin{thm}
  \label{thm:sympmodconv-decomp}
  Let \(\g\) be a reductive Lie algebra. Let \((V,\Omega)\) be a symplectic \(\g\)-module of convex type and \(\{V_1,\ldots,V_n\}\) be a maximal set of pairwise non-equivalent non-trivial simple submodules. Then the following holds:
  \begin{enumerate}
    \item If \(V\) is a faithful \(\g\)-module, then \(\g\) is quasihermitian.
    \item Let \(\Omega_j := \Omega\lvert_{V_j \times V_j}\) and \(m_j\) denote the multiplicity of \(V_j\) in \(V\). Then \((V,\Omega)\) is, as a symplectic \(\g\)-module, equivalent to the orthogonal direct sum \(\bigoplus_{j=1}^n (V_j, \Omega_j)^{m_j}\). For each simple submodule \((W,\Omega')\) with \(W \cong V_j\), we have \((W,\Omega') \cong (V_j,\Omega_j)\).
    \item Let \(\g = \g_0 \oplus \bigoplus_{k=1}^m \g_k\), where \(\g_0\) is compact and \(\g_k\) is hermitian simple for \(1 \leq k \leq m\). For each \(1 \leq j \leq n\) there exists at most one \(k \in \{1,\ldots,m\}\) such that \(\g_k\) acts non-trivially on \(V_j\).
  \end{enumerate}
\end{thm}
\begin{proof}
  (a) Since \(V\) is faithful, the generating invariant cone \(W_V\) (cf.\ \eqref{eq:sympmod-cone}) is pointed. Hence, \(\g\) is admissible and therefore quasihermitian by \cite[Thm.\ VII.3.10]{Ne00}. (b) is shown in \cite[Thm.\ A.VI.5]{Ne00}. For the proof of (c), we refer to \cite[Thm.\ 3.6]{Neu00}.
\end{proof}

\begin{lem}
  \label{lem:sympmodconv-simple-herm-ideal}
  Let \(\g\) be a quasihermitian reductive Lie algebra and let \((V,\Omega)\) be a symplectic \(\g\)-module of convex type. Let \(\fh \subset \g\) be a hermitian simple ideal of \(\g\) such that \(\fh\) acts non-trivially on \(V\). Then there exists a decomposition \(V = V_0 \oplus V_1\) of \(V\) into orthogonal \(\fh\)-submodules such that \(\fh.V_0 = \{0\}\) and \((V_1, \Omega\lvert_{V_1 \times V_1})\) is a symplectic \(\fh\)-module of convex type.
\end{lem}
\begin{proof}
  By \cite[Prop.\ 3.5(1)]{Neu00}, the cone \(W_V \cap \fh\) (cf.\ \eqref{eq:sympmod-cone}) generates \(\fh\). Setting \(V_0 := V_{\mathrm{fix}}\) with respect to the action of \(\fh\) on \(V\) and \(V_1 := \spann(\fh.V)\), the claim follows from Remark \ref{rem:sympmod-conv-defs}.
\end{proof}

\begin{lem}
  \label{lem:sympmodconv-submodules}
  Let \(\g\) be a Lie algebra and let \((V,\Omega)\) be a symplectic \(\g\)-module of convex type. Then every \(\g\)-submodule \(W \subset V\) becomes a symplectic \(\g\)-module of convex type when endowed with the restriction \(\Omega\lvert_{W \times W}\).
\end{lem}
\begin{proof}
  cf. \cite[Lem.\ A.VI.3]{Ne00}.
\end{proof}

\begin{example}
  \label{ex:sympl-mod-conv-sl2}
  Since it will be important later, we discuss the symplectic modules of convex type of the hermitian simple Lie algebra \(\g := \fsl(2,\R)\). As all symplectic \(\g\)-modules \((V,\Omega)\) of convex type are semisimple (cf.\ Theorem \ref{thm:sympmodconv-decomp}), we assume that \(V\) is simple. Then the classification of symplectic modules of convex type of simple Lie algebras (cf.\ \cite[Thm.\ III.15]{Ne94a}) shows that, as a \(\g\)-module, \(V\) is equivalent to the standard representation of \(\fsl(2,\R) \cong \sp(2,\R)\) on \(\R^2\). The symplectic form \(\Omega\) on \(V\) which turns \(V\) into a symplectic module is unique up to equivalence and the choice of sign.

  Suppose now that \(V := \R^2\) is endowed with the standard representation of \(\g\) and let \(\Omega\) denote the standard symplectic form on \(V\). Let \((W,\tilde \Omega)\) be a symplectic \(\g\)-module of convex type. By Theorem \ref{thm:sympmodconv-decomp}, the module \((W,\tilde \Omega)\) is then equivalent to the direct sum \(\bigoplus_{j=1}^n (V, \Omega)\) for some \(n \in \N\). In particular, the signs of the symplectic forms on each submodule are coupled. The corresponding homomorphism \(\rho: \fsl(2,\R) \rightarrow \sp(2n,\R)\) is up to equivalence determined by
  \[\rho(H) = \pmat{\bbone_n & 0 \\ 0 & -\bbone_n} \quad \text{and} \quad \rho(U) = \pmat{0 & \bbone_n \\ -\bbone_n & 0}.\]
  Note that \(\rho\) is an \((H_2)\)-homomorphism with respect to \(\frac{1}{2}U\) and \(\frac{1}{2}\rho(U)\).
\end{example}

Let \((V,\Omega)\) be a symplectic vector space. For \(X \in \End(V)\), we define \(X^\# \in \End(V)\) as the endomorphism on \(V\) determined by
\[\Omega(Xv,w) = \Omega(v,X^\# w) \quad \text{for all } v,w \in V.\]
The map \(X \mapsto X^\#\) is an involutive anti-automorphism on the algebra \(\End(V)\).

\begin{prop}
  \label{prop:sympmodconv-endo}
  Let \(\g\) be a Lie algebra and let \((V,\Omega)\) be a symplectic \(\g\)-module.
  \begin{enumerate}
    \item The algebra \(\End_\g(V)\) of \(\g\)-endomorphisms of \(V\) is \(\#\)-invariant.
    \item Suppose that \(V\) is simple and of convex type. Then \(\End_\g(V)\) is isomorphic to \(\R\), \(\C\), or \(\H\) and \(\#\) acts trivially on \(\End_\g(V)\) if \(\End_\g(V) = \R \id_V\) and by the natural conjugation on \(\C\), respectively \(\H\), if \(\End_\g(V) \cong \C\), respectively \(\End_\g(V) \cong \H\). In particular, we have \(\End_\g(V)^\# = \R \id_V\).
  \end{enumerate}
\end{prop}
\begin{proof}
  (a) If \(A \in \End_\g(V)\), then we have for all \(x \in \g\) and \(v,w \in V\) that
  \[\Omega(v,A^\# x.w) = \Omega(Av,x.w) = -\Omega(x.(Av),w) = -\Omega(Ax.v,w) = \Omega(v,x.(A^\#)w),\]
  so that \(A^\# \in \End_\g(V)\).

  (b) That \(\End_\g(V)\) is isomorphic to \(\R\), \(\C\), or \(\H\) follows from Schur's Lemma. For the proof of the second statement, we refer to \cite[Lem.\ II.7]{Ne94a}.
\end{proof}

\begin{lem}
  \label{lem:sympmodconv-spec-skewsym-end}
  Let \(\g\) be a Lie algebra and let \((V,\Omega)\) be a symplectic \(\g\)-module of convex type. Then \(\spec(X) \subset i\R\) for each \(X \in \End_\g(V)^{-\#}\).
\end{lem}
\begin{proof}
  Let \(X \in \End_\g(V)^{-\#} = \End_\g(V) \cap \sp(V,\Omega)\). Since \((V,\Omega)\) is of convex type, there exists \(y \in \g\) such that the Hamiltonian function
  \[H_y : V \rightarrow \R, \quad H_y(v) := \Omega(y.v,v),\]
  is positive definite. In particular, we have \(\lim_{\|v\| \rightarrow \infty} H_y(v) = \infty\), so that the level sets of \(H_y\) are bounded and hence compact. Let \(v \in V\). Then \(H_y(e^{tX}v) = H_y(v)\) for all \(t \in \R\) implies that the set \(e^{\R X}v\) is bounded. By the principle of uniform boundedness, this means that \(e^{\R X}\) is bounded, so that \(\oline{e^{\R X}}\) is compact in \(\GL(V)\). In particular, we have \(\spec(X) \subset i\R\).
\end{proof}

\begin{example}
  \label{ex:jacobi-alg}
  The following is a prototypical example of an admissible Lie algebra: Let \((V,\Omega)\) be a symplectic vector space. Then \(V\) is a symplectic \(\sp(V,\Omega)\)-module of convex type with the canonical action of \(\sp(V,\Omega)\) on \(V\).
  Using Spindler's construction, we define the Jacobi algebra by
  \[\hsp(V,\Omega) := \g(\sp(V,\Omega),V,\R,\Omega).\]
  Its radical can be identified with the Heisenberg algebra \(\heis(V,\Omega) := V \times \R\) with the bracket
  \[[(v,z),(v',z')] := (0,\Omega(v,v')).\]
  The Jacobi algebra can be identified with the Lie algebra \(\Pol_{\leq 2}(V)\) of polynomials of degree at most \(2\) endowed with the Poisson bracket via
  \begin{equation*}
    \varphi : \hsp(V,\Omega) \rightarrow \Pol_{\leq 2}(V), \quad \varphi(w,z,x)(v) := \frac{1}{2}\Omega(xv,v) + \Omega(w,v) + z,
  \end{equation*}
  (cf.\ \cite[Prop.\ A.IV.15]{Ne00}). The pointed invariant closed convex cone
  \begin{equation*}
    \hsp(V,\Omega)_+ := \{x \in \hsp(V,\Omega) : \varphi(x) \geq 0\}
  \end{equation*}
  generates \(\hsp(V,\Omega)\).
  For every \(H\)-element \(J \in \hsp(V,\Omega)_+\) of \(\sp(V,\Omega)\) and every constant \(c > 0\), the element \(J + c\) is an interior point of \(\hsp(V,\Omega)_+\) (cf.\ \cite[Ex.\ VIII.2.3]{Ne00}).
\end{example}

We can easily define an analog of the Jacobi algebra in Example \ref{ex:jacobi-alg} with larger center: For real vector spaces \(V\) and \(\fz\) and a skew-symmetric bilinear map \(\beta: V \times V \rightarrow \fz\), we define the \emph{symplectic Lie algebra on \((V,\beta)\)} by
\[\sp(V,\beta) := \{x \in \End(V) : (\forall v,w \in V)\, \beta(xv, w) + \beta(v, xw) = 0\}.\]
We now define the \emph{generalized Jacobi algebra} \(\hsp(V,\beta) := \g(\sp(V,\beta), V, \fz, \beta)\) using Spindler's construction and the natural action of \(\sp(V,\beta)\) on \(V\).

\begin{rem}
  \label{rem:gen-symp-alg-char}
  There is an alternative way to characterize the symplectic Lie algebra \(\sp(V,\beta)\) for \(\beta: V \times V \rightarrow \fz\): Suppose that there exists \(f \in \fz^*\) such that \(\Omega := f \circ \beta\) is non-degenerate, i.e.\ a symplectic form of \(V\). Then, for every \(g \in \fz^*\), there exists a unique linear map \(\Psi_g \in \End(V)\) such that
  \[(g \circ \beta)(v,w) = \Omega(\Psi_g v,w) \quad \text{for all } v,w \in V.\]
  The maps \(\Psi_g\) are symmetric with respect to \(\Omega\) and we have \([\Psi_g, \sp(V,\beta)] = \{0\}\) by \cite[Lem.\ 5.9]{Neu00}.

  Let \(D \subset \fz^*\) be a generating subset of \(\fz^*\). Then we even have
  \[\sp(V,\beta) = \{x \in \sp(V,\Omega) : (\forall g \in D) [x,\Psi_g] = 0\}.\]
  The arguments above show that the inclusion ``\(\subset\)'' holds. Conversely, suppose that \(x \in \sp(V,\Omega)\) commutes with \(\Psi_g\) for all \(g \in D\). Then we have
  \[(g \circ \beta)(xv,w) = \Omega(\Psi_g xv,w) = \Omega(x \Psi_g v,w) = -\Omega(\Psi_g v, xw) = -(g \circ \beta)(v, xw)\]
  for all \(v,w \in V\). Since \(D\) is generating in \(\fz^*\), this shows that \(x \in \sp(V,\beta)\).
\end{rem}

The following theorem shows that, for every admissible Lie algebra \(\g = \g(\fl, V, \fz, \beta)\), the corresponding generalized Jacobi algebra \(\hsp(V,\beta)\) is also admissible. It is essentially a reformulation of \cite[Thm.\ V.10]{Ne94a} and the arguments in \cite[Ch.\ III \S 6]{Sa80}.

\begin{thm}
  \label{thm:adm-gen-jacobi-alg}
  Let \(\g = \g(\fl,V,\fz,\beta)\) be an admissible Lie algebra. Then the following holds:
  \begin{enumerate}
    \item The Lie algebra \(\sp(V,\beta)\) is reductive and admissible.
    \item Let \(f \in \fz^*\) such that, for \(\Omega := f \circ \beta\), the space \((V,\Omega)\) is a symplectic \(\fl\)-module of convex type. Then the inclusion \(\sp(V,\beta) \hookrightarrow \sp(V,\Omega)\) is an \((H_2)\)-homomorphism.
    \item The Lie algebra \(\hsp(V,\beta) = \g(\sp(V,\beta), V, \fz, \beta)\), where \(\sp(V,\beta)\) acts canonically on \(V\), is admissible.
  \end{enumerate}
\end{thm}
\begin{proof}
  We choose a compactly embedded Cartan subalgebra \(\ft_\fl \subset \fl\) of \(\fl\). By Theorem \ref{thm:spindlers-constr}, the subalgebra \(\ft := \{0\} \times \fz \times \ft_\fl\) is then a compactly embedded Cartan subalgebra of \(\g\). We denote by \(\Delta := \Delta(\g_\C, \ft_\C)\) the corresponding set of roots. Let \(\Delta^+\) be an adapted system of positive roots such that the cone \(\Cmin\) is pointed (cf.\ \cite[Thm.\ VIII.2.1]{Ne00}). Then the cone \(C_{\mathrm{min}, \fz} := \Cmin \cap \fz\) is also pointed, so that the dual cone \(C_{\mathrm{min}, \fz}^*\) is generating in \(\fz\).

  Since the maximal nilpotent ideal of \(\g\) is given by \(V \times \fz\), we have by \cite[Prop.\ VII.2.5]{Ne00} that \(V_\C = \sum_{\alpha \in \Delta_r} \g_\C^\alpha\).
  Let \(V_\pm := \sum_{\alpha \in \Delta_r^\pm} \g_\C^\alpha\).
  Then \(\oline{V_\pm} = V_\mp\) and \(V_\C = V_- \oplus V_+\).
  We define a complex structure \(I_\C\) on \(V_\C\) by defining \(I_\C v := \pm iv\) for \(v \in V_\pm\).
  Then we obtain a complex structure \(I\) on \(V\) by using the isomorphism
  \[V_+ \rightarrow V, \quad v \mapsto v + \oline{v}.\]
  We now show that \(I \in \sp(V,\beta)\): For \(\alpha,\beta \in \Delta_r^\pm\), we have
  \[[\g_\C^\alpha,\g_\C^\beta] \subset \g_\C^{\alpha + \beta} \cap [V,V] \subset V \cap \fz(\g) = \{0\},\]
  so that \(\beta(V_\pm, V_\pm) = \{0\}\). Thus, for \(x,y \in V_+\), we have
  \[\beta(I(x+\oline{x}), y + \oline{y}) = \beta(ix - i\oline{x}, y + \oline{y}) = i(\beta(x,\oline{y}) - \beta(\oline{x},y)) = -\beta(x+\oline{x}, I(y + \oline{y})),\]
  so that \(I \in \sp(V,\beta)\). Let \(g \in (C_{\mathrm{min}, \fz}^*)^o\) and let \(\Omega_g := g \circ \beta\). Then
  \[\beta(I(x+\oline{x}), x + \oline{x}) = i(\beta(x,\oline{x}) - \beta(\oline{x},x)) = 2i[x,\oline{x}] \in C_{\mathrm{min}, \fz}\]
  for all \(x \in V^+\) shows that \(\Omega_g(Iv,v) > 0\) for all \(v \in V\), i.e.\ \(I\) is contained in the interior of the up to sign unique pointed generating invariant closed convex cone \(W_g \subset \sp(V,\Omega_g)\). In particular, the invariant cone \(W_g \cap \sp(V,\beta)\) is pointed and generating in \(\sp(V,\beta)\), which shows that \(\sp(V,\beta)\) is admissible.
  Since \(I\) is a complex structure of \(V\) with \(I \in W_g\), the map \(\theta(x) := IxI^{-1}, x \in \sp(V,\Omega_g)\), is a Cartan involution of \(\sp(V,\Omega_g)\) (cf.\ \cite[Ch.\ II \S 7]{Sa80}).
  By using the alternative characterization of \(\sp(V,\beta)\) from Remark \ref{rem:gen-symp-alg-char}, we see that the subalgebra \(\sp(V,\beta)\) is \(\theta\)-invariant, so that it is reductive by \cite[Cor.\ 1.1.5.4]{Wa72}.
  This proves (a).
  We also note that \(I \in \fz(\sp(V,\Omega_g)^\theta)\), which shows that \(I\) is an \(H\)-element of \(\sp(V,\Omega_g)\).
  Therefore, it is also an \(H\)-element of \(\sp(V,\beta)\), so that the inclusion of \(\sp(V,\beta)\) into \(\sp(V,\Omega_g)\) is an \((H_2)\)-homomorphism.
  This proves (b) because, for every linear functional \(f \in \fz^*\) with the property that \((V, f \circ \beta)\) is a symplectic \(\fl\)-module of convex type, we can choose \(\Delta^+\) in such a way that \(f\) is contained in the interior of \(C_{\mathrm{min}, \fz}^*\) by \cite[Prop.\ VIII.2.9(i)]{Ne00}.

  (c) follows from Theorem \ref{thm:spindler-adm} because \(\sp(V,\beta)\) is admissible by (a) and therefore quasihermitian by \cite[Thm.\ VII.3.10]{Ne00}, and because \((V, g \circ \beta)\) is a symplectic \(\sp(V,\beta)\)-module of convex type.
\end{proof}

\subsection{Reductive complements in admissible Lie algebras}
\label{sec:adm-reducompl}

As we have seen in Section \ref{sec:lie-adm}, every admissible Lie algebra \(\g\) can be written as a semidirect product \(\g = \fu \rtimes \fl\), where \(\fu\) is the maximal nilpotent ideal of \(\g\) and \(\fl\) is a reductive subalgebra of \(\g\). Moreover, this decomposition has the additional properties that \(\fl\) is quasihermitian and \(\fu\) can be written as \(\fu = V + \fz(\g)\), where \(V := [\fl,\fu]\) can be regarded as a symplectic \(\fl\)-module of convex type.

Our first goal in this section is to show that every decomposition of the form \(\g = \fu \rtimes \fl'\), where \(\fl'\) is a subalgebra of \(\g\), has these additional properties. Since this will be important later, we will also show for every semisimple derivation \(D \in \g\) that induces a 3-grading on \(\g\) the existence of a \(D\)-invariant decomposition \(\g = \fu \rtimes \fl\) such that \(D(\fz(\fl)) = \{0\}\).

\begin{lem}
  \label{lem:adm-semprod-solv}
  Let \(\g\) be a solvable admissible Lie algebra. Let \(\fu \subset \g\) be the maximal nilpotent ideal of \(\g\) and let \(\fa \subset \g\) be an abelian subalgebra of \(\g\) such that \(\g = \fu \rtimes \fa\). Then \(\fa\) is compactly embedded in \(\g\) and \(\ft := \fz(\g) + \fa\) is a compactly embedded Cartan subalgebra of \(\g\).
\end{lem}
\begin{proof}
  By Theorem \ref{thm:spindlers-constr}, there exists a compactly embedded abelian subspace \(\ft_r' \subset \g\) with \(\ft_r' \cap \fu = \{0\}\) such that \(\ft' := \fz(\g) \oplus \ft_r'\) is a compactly embedded Cartan subalgebra with \(\g = \fu + \ft_r'\).
  By \cite[Ch.\ VII, \S 2, no. 1, Cor.\ 1]{Bou05}, there exists \(y \in \ft_r'\) such that \(\ft'\) is the nilspace of \(\ad(y)\), which is also the centralizer of \(y\) because \(\ft'\) is abelian.
  Since \(\fu + \ft_r' = \fu + \fa\), there exists \(x \in \fa\) and \(u \in \fu\) such that \(x = u + y\).

  Using Theorem \ref{thm:spindlers-constr}, we decompose \(\fu\) into \(\fu = V + \fz(\g)\), where \(V := [\ft', \fu]\).
  For \(u' \in \fu\), we then have \([y,u'] \in V\) and \([u,u'] \in \fz(\g)\).
  Since \(V \cap \fz(\g) = \{0\}\) and \(\ad(y)\lvert_V\) is injective, this shows that \(\ker(\ad_\fu(x)) = \fz(\g)\) and thus \(\ker(\ad(x)) = \fz(\g) + \fa\).
  The above argument also shows that the nilspace of \(x\) is \(\fz(\g) + \fa\) because \(\ad(x)(\g) \subset \fu\). 
  Since \(\fa \cong \g/\fu \cong \ft_r'\), this shows that \(x\) is regular.
  In particular, \(\ft = \fz(\g) + \fa\) is a Cartan subalgebra of \(\g\).
  By \cite[Ch.\ VII, \S 3, no.\ 4, Thm.\ 3]{Bou05}, the Cartan subalgebras \(\ft\) and \(\ft'\) are conjugate under an inner automorphism of \(\g\).
  Hence, \(\ft\) is compactly embedded in \(\g\) because \(\ft'\) is compactly embedded in \(\g\).
\end{proof}

\begin{thm}
  \label{thm:adm-semprod}
  Let \(\g\) be an admissible Lie algebra and let \(\fu\) be the maximal nilpotent ideal of \(\g\). Let \(\fl \subset \g\) be a subalgebra such that \(\g = \fu \rtimes \fl\). Then the following holds:
  \begin{enumerate}
    \item \(\fl\) is reductive and quasihermitian.
    \item If \(\ft_\fs \subset \fs := [\fl,\fl]\) is a compactly embedded Cartan subalgebra of \(\fs\), then \(\ft := \fz(\g) + \fz(\fl) + \ft_\fs\) is a compactly embedded Cartan subalgebra of \(\g\).
  \end{enumerate}
\end{thm}
\begin{proof}
  (a) By Theorem \ref{thm:spindlers-constr}, there exists a reductive subalgebra \(\fl' \subset \g\) with \(\g = \fu \rtimes \fl'\) which is quasihermitian by Theorem \ref{thm:spindler-adm}. Now \(\fl \cong \g/\fu \cong \fl'\) implies that \(\fl\) is also reductive and quasihermitian.

  (b) Let \(\ft_r := \fz(\fl)\).
  We first note that \(\rad(\g) = \fu + \ft_r\), so that \(\fs\) is a Levi-complement of \(\g\).
  The compactly embedded subalgebra \(\ft_\fs\) is also compactly embedded in \(\g\) because \(\fs\) is semisimple (cf.\ \cite[Lem.\ I.9]{KN96}).
  Since compactly embedded Cartan subalgebras and Levi complements are conjugate under inner automorphisms, we can use \cite[Prop.\ VII.1.9(i)]{Ne00} to obtain a compactly embedded Cartan subalgebra \(\ft'\) of \(\g\) with \(\ft' = (\ft' \cap \rad(\g)) \oplus \ft_\fs\) and \([\ft' \cap \rad(\g), \fs] = \{0\}\).
  Let \(\fa \subset \ft' \cap \rad(\g)\) be a vector space complement of \(\fz(\g)\).
  Then \(\fl' := \fa + \fs\) is a reductive Lie algebra and we have
  \[\fu \rtimes (\ft_r + \fs) = \g = \fu \rtimes (\fa + \fs) \quad \text{and} \quad \fu + \ft_r = \rad(\g) = \fu + \fa\]
  (cf.\ \cite[Prop.\ VII.1.9(ii)(c)]{Ne00}).
  In particular, we have
  \[\g_r := \fu + \ft_r + \ft_\fs = \fu + \fa + \ft_\fs,\]
  and this Lie algebra is solvable and admissible because \(\ft' = \fz(\g) + \fa + \ft_\fs\) is compactly embedded in \(\g_r\).
  Now Lemma \ref{lem:adm-semprod-solv} shows that \(\ft\) is also a compactly embedded Cartan subalgebra of \(\g_r\).
  In particular, \(\ft\) is a compactly embedded Cartan subalgebra of \(\g\) because \([\ft_r, \fs] = \{0\}\).
\end{proof}

\begin{cor}
  \label{cor:adm-semprod-spindler}
  Let \(\g\) be an admissible Lie algebra and let \(\fu\) be the maximal nilpotent ideal of \nolinebreak\(\g\). Let \(\fl \subset \g\) be a subalgebra such that \(\g = \fu \rtimes \fl\) and let \(\ft_\fl \subset \fl\) be a compactly embedded Cartan subalgebra. Then \(V := [\ft_\fl, \fu]\) is an \(\fl\)-module with respect to the adjoint representation and we have \(\fz_V(\ft_\fl) = \{0\}\). Moreover, we have \(\g \cong \g(\fl, V, \fz(\g), \beta)\), where
  \[\beta: V \times V \rightarrow \fz(\g), \quad \beta(x,y) := [x,y].\]
\end{cor}
\begin{proof}
  By Theorem \ref{thm:adm-semprod}, the subalgebra \(\fl\) is reductive and quasihermitian and \(\ft := \fz(\g) + \ft_\fl\) is a compactly embedded Cartan subalgebra of \(\g\).
  Since \([\fl,\fl]\) is a Levi complement of \(\g\), \cite[Prop.\ VII.1.9]{Ne00} therefore implies that \(V\) is \(\fl\)-invariant and that \(\fu = V + \fz(\g)\).
  We also have \(V \cap \fz(\g) = \{0\}\) because \(\fz_\g(\ft) = \ft\).
  Hence, \(\g \cong \g(\fl,V,\fz(\g),\beta)\) in terms of Spindler's construction.
  Note that \(\beta\) can be defined in the above way because we have \([\fu,\fu] \subset \fz(\g)\) by Theorem \ref{thm:spindlers-constr}.
\end{proof}

\begin{lem}
  \label{lem:der-levi-invcomp}
  Let \(\g\) be a real or complex Lie algebra and let \(D \in \der(\g)\) such that
  \[\g = \g_{-1}(D) \oplus \g_0(D) \oplus \g_1(D).\]
  Then there exists a \(D\)-invariant Levi complement in \(\g\).
\end{lem}
\begin{proof}
  Let \(A := (e^{tD})_{t \in \R} \subset \Aut(\g)\) be the one-parameter subgroup generated by \(D\). Then \(\g\) is a semisimple \(A\)-module because \(D\) is diagonalizable. Hence, there exists an \(A\)-invariant Levi complement in \(\g\) by \cite[Prop.\ I.2]{KN96}, which is also \(D\)-invariant.
\end{proof}

The following theorem makes use of algebraic groups and in particular algebraic Lie algebras. We refer to Appendix \ref{sec:app-alg-grp} for an overview of all the facts that are used here.

\begin{thm}
  \label{thm:adm-deriv-3grad-reducompl}
  Let \(\g\) be an admissible Lie algebra and let \(\fu\) be the maximal nilpotent ideal of \(\g\). Let \(D \in \der(\g)\) such that \(\g = \g_{-1}(D) \oplus \g_0(D) \oplus \g_1(D)\). Then there exists a \(D\)-invariant reductive quasihermitian subalgebra \(\fl \subset \g\) such that \(\g = \fu \rtimes \fl\) and \(D(\fz(\fl)) = \{0\}\).
\end{thm}
\begin{proof}
  \textbf{Step 1:} Let \(G\) be a 1-connected Lie group with Lie algebra \(\g\).
  By Lemma \ref{lem:der-levi-invcomp}, there exists a \(D\)-invariant Levi complement \(\fs\) of \(\g\).
  Let \(S \subset G\), resp.\ \(U \subset G\), be the integral subgroup of \(G\) with Lie algebra \(\fs\), resp.\ \(\fu\).
  Moreover, we consider the closed subgroup \(S_D := \Ad(S)e^{\R D} \subset \Aut(\g)\).

  \textbf{Step 2:} The Lie algebra \(\L(S_D) \cong \ad_\g \fs \rtimes \R D\) is algebraic by Lemma \ref{lem:subalg-algebraic-liebrack} because \(\fs\) is perfect and \(\g_D := a(\ad \g) \rtimes \R D \supset \ad \g \rtimes \R D\) is algebraic by Lemma \ref{lem:lie-adm-deriv-alg}. Here, the derivation \(D\) acts on \(\ad \g\) by \(\ad D\). Thus, there exist algebraic subgroups \(F^A\) and \(S_D^A\) of \(\Aut(\g)\) with \(\L(F^A) = \g_D\) and \(\L(S_D^A) = \L(S_D)\). Moreover, \(S_D^A\) is a reductive group with \(S_D^A \subset F^A\) (cf.\ \cite[IV.2 Thm.\ 2.2]{Ho81}). In particular, there exists a maximal reductive subgroup \(F_r^A \supset S_D^A\) of \(F^A\) such that \(F^A = F_u^A \rtimes F_r^A\) is a Levi decomposition of \(F^A\), where \(F_u^A\) denotes the unipotent radical of \(F^A\). Let \(\ff_r := \L(F_r^A)\).

  The Lie algebra \(\L(F_u^A)\) is the maximal nilpotent ideal \(\ad(\fu)\) of \(a(\ad(\g))\) (cf.\ Lemma \ref{lem:lie-adm-deriv-alg}). Therefore, the 1-component of \(F_u^A\) is given by \(\Ad(U)\), so that \(\Ad(G) \subset F^A\) is adapted to the Levi decomposition of \(F^A\) in the sense that
  \[\oline{\Ad(G)} = \Ad(U) \rtimes L', \quad \text{where} \quad L' := (F_r^A \cap \oline{\Ad(G)}),\]
  and the closure is taken in \(\Aut(\g)\).
  On the level of Lie algebras, we obtain a semidirect sum \(\ad \g = \ad(\fu) \rtimes \fl'\), where \(\fl' := \L(L') \cap \ad(\g)\). The subalgebra \(\ad(\g)\) is an ideal in \(\der(\g)\), so that \(\fl' = \ad(\g) \cap \ff_r\) is reductive because it is an ideal of \(\ff_r\). In particular, we have \(\fz(\fl') = \fz(\ff_r) \cap \ad(\g)\).

  The center of \(F_r^A\) commutes with \(e^{\R D} \subset S_D^A \subset F_r^A\), so that we have in particular

  \[e^{t D} \Ad(\exp(x)) e^{-tD} = \Ad(\exp(x)) \quad \text{ for all } x \in \fz(\fl'), t \in \R,\]
  and therefore \([D, e^{\R \ad x}] = \{0\}\), which leads to \([D,\fz(\fl')] = \{0\}\). The commutator algebra \([\fl', \fl']\) is a semisimple subalgebra of \(\g\), so that the maximality of the Levi complement \(\ad \fs \subset \fl'\) leads to \(\fl' = \ad \fs \oplus \fz(\fl')\).

  \textbf{Step 3:} Let \(\fa := \ad_\g^{-1}(\fz(\fl')) \subset \g\). Then
  \[\g = \ad^{-1}(\ad(\fu)) + \ad^{-1}(\fl') = \fu + \fa + \fs.\]
  Moreover, \(\fa\) commutes with \(\fs\): To see this, let \(x \in \fa\). Then \(\ad(x) \in \fz(\fl')\) implies that \([\fa, \fs] \subset \fz(\g)\). But since \([\fs, \fs] = \fs\), we even have \([\fa,\fs] = \{0\}\). Hence, \(\fu + \fa\) is solvable ideal in \(\g\) and \(\g/(\fu + \fa) \cong \fs\) shows that \(\rad(\g) = \fu + \fa\).

  \(\fa\) is also abelian: Recall that \(\ff_r = \L(F_r^A) \subset \der(\g)\) is a reductive algebraic Lie subalgebra. Thus, \(\g\) is a semisimple \(\ff_r\)-module under the canonical action of \(\ff_r\) by derivations. In particular, \(\fa\) is a semisimple \(\ff_r\)-submodule with \(\ff_r.\fa \subset \fz(\g) \subset \fa\) because \(\ad(\fa) = \fz(\fl') \subset \fz(\ff_r)\). Since \(\fz(\g)\) is also a \(\ff_r\)-submodule, there exists a module complement \(\fa' \subset \fa\) such that \(\fa = \fa' \oplus \fz(\g)\) as \(\ff_r\)-submodules. Now
  \[\ff_r.\fa' \subset \fa' \cap \fz(\g) = \{0\} \quad \text{implies that} \quad [\fa,\fa'] = \ad(\fa)(\fa') \subset \ff_r.(\fa') = \{0\},\]
  which shows that \([\fa,\fa] = \{0\}\), i.e.\ that \(\fa\) is abelian.

  By construction of \(\fa\), we have \(D(\fa) \subset \fz(\g)\). Using the fact that \(D\) is semisimple, we can find a vector space complement \(\ft_r\) of \(\fz(\g)\) in \(\fa\) such that \(D(\ft_r) = \{0\}\). Now \(\g\) can be written as a semidirect sum
  \[\g = \fu \rtimes \fl \quad \text{with} \quad \fl := \fs \oplus \ft_r.\]
  \textbf{Step 4:} It remains to conclude that \(\fl\) is quasihermitian and reductive: This follows from Theorem \ref{thm:adm-semprod} because \(\g\) is admissible and \(\g = \fu \rtimes \fl\).
\end{proof}

\section{Semisimple derivations of admissible Lie algebras}
\label{sec:liewedge-adm}

Let \(\g\) be an admissible Lie algebra.
By Theorem \ref{thm:spindlers-constr} and Theorem \ref{thm:spindler-adm}, we then have \(\g = \g(\fl,V,\fz,\beta)\), where \(\fl\) is a reductive quasihermitian Lie algebra, \(V\) is an \(\fl\)-module, and \(\beta: V \times V \rightarrow \fz\) is an \(\fl\)-invariant skew-symmetric bilinear map.
Moreover, there exists \(f \in \fz^*\) such that \((V,f \circ \beta)\) is a symplectic \(\fl\)-module of convex type.
Let \(\rho: \g \rightarrow \sp(V,f \circ \beta)\) be the module homomorphism associated with \((V,f \circ \beta)\).
Then there exists a natural homomorphism
\begin{equation}
  \label{eq:spindler-jacobi-hom}
  \Phi_f: \g(\fl,V,\fz,\beta) \rightarrow \hsp(V,f \circ \beta), \quad (v,z,x) \mapsto (v,f(z),\rho(x)),
\end{equation}
and the invariant closed convex cone \(W_f := \Phi_f^{-1}(\hsp(V,f \circ \beta)_+)\) is generating in \(\g\) (cf.\ \cite[Thm.\ VIII.2.7]{Ne00}).

In this section, we characterize the derivations \(D \in \der(\g)\) with the property that 
\begin{itemize}
  \item \(\g = \g_{-1}(D) \oplus \g_0(D) \oplus \g_1(D)\) and
  \item \(\g_{\pm 1}(D) \cap W_f\) span \(\g_{\pm 1}(D)\).
\end{itemize}

\begin{rem}
  \label{rem:spindler-cones-unitary-rep}
  (a) An alternative way to characterize the cone \(W_f\) is the following: Suppose that \(\g = \g(\fl,V,\fz,\beta)\) in terms of Spindler's construction.
Let \(\ft_\fs\) be a compactly embedded Cartan subalgebra of \(\fs := [\fl,\fl]\).
Then \(\ft := \{0\} \times \fz \times \ft_\fl\), where \(\ft_\fl := \fz(\fl) \oplus \ft_\fs\), is a compactly embedded Cartan subalgebra of \(\g\) (cf.\ Theorem \ref{thm:spindlers-constr}).
In particular, we have \(\g = [\ft,\g] \oplus \ft\) as a vector space.
Let \(f \in \fz^*\) be such that \((V,f \circ \beta)\) is a symplectic \(\fl\)-module of convex type.
Then we can extend \(f\) to a linear functional on \(\g\) by \(0\) using the decomposition of \(\g\) provided by \(\ft\).
The cone \(W_f\) then coincides with the dual cone \(\cO_f^*\), where \(\cO_f \subset \g^*\) is the coadjoint orbit of \(f\) (cf.\ \cite[Prop.\ VIII.2.4]{Ne00}).
The cone \(W_f\) therefore depends on the choice of an extension of \(f \in \fz^*\) to \(\g\), which is in turn determined by the choice of the Spindler data \((\fl,V,\fz,\beta)\).
Note that, while changing the compactly embedded Cartan subalgebra \(\ft_\fs\) of \(\fs\) does change the extension of \(f\) to \(\g\), it does not change the coadjoint orbit \(\cO_f\) because compactly embedded Cartan subalgebras are conjugate under inner automorphisms.

(b) In the introduction, we explained that we are mainly interested in classifying the Lie algebras \(\g(W,\tau,h)\) for those \(W\) which are positive cones of antiunitary representations.
This ties into our consideration of the cones \(W_f\) in the following way: Let \(G\) be a connected Lie group with Lie algebra \(\g\), where \(\g\) is admissible. Let \((U,\cH)\) be a strongly continuous irreducible unitary representation of \(G\) on a complex Hilbert space \(\cH\) such that the kernel of \(U\) is discrete and the positive cone \(C_U \subset \g\) is generating.
Then \((U,\cH)\) is a highest weight representation by \cite[Thm.\ X.3.9]{Ne00} in the sense of \cite[Def.\ X.2.9]{Ne00}. Let \(\ft \subset \g\) be a compactly embedded Cartan subalgebra of \(\g\) and let \(\lambda \in i\ft^* \subset \ft_\C\) be the highest weight of \((U,\cH)\).
Then we have \(C_U = \cO_{-i\lambda}^*\) by \cite[Thm.\ X.4.1]{Ne00}.
Note that \(f := -i\lambda \in \ft^*\) by \cite[Thm.\ VII.2.2]{Ne00}.
Moreover, there exists a positive adapted system \(\Delta^+\) such that \(f\) is contained in the interior of the cone \(C_{\min}^*\).
Let \(\g = \g(\fl,V, \fz, \beta)\) be a description of \(\g\) in terms of Spindler's construction that is adapted to \(\ft\) (cf.\ Theorem \ref{thm:spindlers-constr}). Then we can decompose \(f\) into \(f = f_Z + f_L\), where \(f_Z\) vanishes on \(\ft \cap \fl\) and \(f_L\) vanishes on \(\fz(\g)\). In particular, we have \(f_Z \in (C_{\min, \fz}^*)^o\), so that \((V, f \circ \beta)\) is a symplectic \(\fl\)-module of convex type.
Furthermore, let \(U\) and \(L\) be 1-connected Lie groups with \(\L(U) = \fu\) and \(\L(L) = \fl\). Then we have
\[\cO_f = G.f = \cO_{f_Z} + \cO_{f_L} = (U.f_Z) + (L.f_L)\]
by \cite[Thm.\ VIII.1.3]{Ne00}, so that \(\cO_{f_Z}^* \cap \cO_{f_L}^* \subset \cO_f^*\). From \(\cO_{f_L} = L.f_L\), we derive that \(\cO_{f_L}(\fu) = \{0\}\), so that \(\fu \subset \cO_{f_L}^*\).
Moreover, \(\cO_{f_Z}^* \cap \cO_{f_L}^* \cap \ft\) contains the cone \(C_{\min}\) for some positive adapted system (cf.\ \cite[Thm.\ VIII.1.19]{Ne00}), which implies in particular that \((\cO_{f_Z}^* \cap \cO_{f_L}^*) \cap \fl\) is generating.

This argument shows that, in the context of our classification problem, it suffices to consider the cones \(\cO_{f_Z}^*\), since they ``differ'' from the intersections \(\cO_{f_Z}^* \cap \cO_{f_L}^*\) only by elements in the reductive subalgebra \(\fl\), which can be controlled rather easily using the results from \cite{Oeh20}.
\end{rem}

\begin{rem}
  \label{rem:spindler-jacobi-hom-reducomp}
  The homomorphism \(\Phi_f\) and therefore also the cone \(W_f\) depend on the choice of the reductive complement of the maximal nilpotent ideal.
  Therefore, we need a better understanding of how these cones change if the reductive complement changes.

  Let \(\g\) be an admissible Lie algebra and let \(\g = \fu \rtimes \fl = \fu \rtimes \tilde \fl\), where \(\fu\) is the maximal nilpotent ideal of \(\g\) and \(\fl\) and \(\tilde \fl\) are reductive quasihermitian Lie algebras. By fixing compactly embedded Cartan subalgebras \(\ft_\fs\) and \(\tilde \ft_\fs\) of \([\fl,\fl]\) and \([\tilde \fl,\tilde \fl]\), we obtain compactly embedded Cartan subalgebras
  \[\ft := \fz(\g) + \fz(\fl) + \ft_\fs \quad \text{and} \quad \tilde \ft := \fz(\g) + \fz(\tilde \fl) + \tilde \ft_\fs\]
  of \(\g\) (cf.\ Theorem \ref{thm:adm-semprod}). These subalgebras then determine submodules \(V = [\ft,\fu]\), resp.\ \(\tilde V = [\tilde \ft, \fu]\), for \(\fl\), resp.\ \(\tilde \fl\), acting on \(\fu\) (cf.\ Corollary \ref{cor:adm-semprod-spindler}).

  Let \(\beta := [\cdot, \cdot]\lvert_{V \times V}\) and \(f \in \fz^* := \fz(\g)^*\) be such that \((V, f \circ \beta)\) is a symplectic \(\fl\)-module of convex type.
  Since compactly embedded Cartan subalgebras \(\ft\) are conjugate under inner automorphisms of \(\g\) (cf.\ \cite[Thm.\ VII.1.4]{Ne00}) and \(\ft\)-invariant Levi complements are uniquely determined by \(\ft\) (cf.\ \cite[Prop.\ VII.2.5]{Ne00}), we can conjugate \(\tilde \fl\) by an inner automorphism \(\varphi\) of \(\g\) in such a way that \(\varphi(\tilde \fl) \subset [\fl,\fl] + \ft\).
  This conjugation does not change the invariant cone we obtain from the homomorphism \eqref{eq:spindler-jacobi-hom}. Hence, we assume now that \(\tilde \fl \subset [\fl,\fl] + \ft\). Note that, since \([\tilde \fl, \tilde \fl] = [\fl, \fl]\), the only freedom we have in choosing \(\tilde \fl\) under this assumption is that of a complement of \(\fz(\g)\) in \(\ft \cap \rad(\g)\).
  Let
  \[\Phi_f : \g(\fl, V, \fz, \beta) \rightarrow \hsp(V, f \circ \beta), \quad \text{resp.} \quad \tilde\Phi_f: \g(\tilde \fl, \tilde V, \fz, \beta) \rightarrow \hsp(\tilde V, f \circ \tilde\beta)\]
  be the homomorphisms from \eqref{eq:spindler-jacobi-hom}, where \(\tilde \beta := [\cdot, \cdot]\lvert_{\tilde V \times \tilde V}\).
  Note that the assumption that \(\tilde \fl \subset [\fl,\fl] + \ft\) implies that \(V = \tilde V\).
  The canonical quotient homomorphism \(\psi: [\fl,\fl] + \ft \rightarrow ([\fl,\fl] + \ft)/\fz(\g)\) yields an isomorphism \(\tilde \psi := (\psi\lvert_{\tilde \fl})^{-1} \circ \psi\lvert_\fl\) (cf.\ \cite[Rem.\ V.4]{Ne94a}), which in turn yields an isomorphism
  \[\Psi : \g(\fl, V, \fz, \beta) \rightarrow \g(\tilde \fl, \tilde V, \fz, \tilde \beta), \quad (v, z, x) \mapsto (v, z, \tilde\psi(x))\]
  with \(\Phi_f = \tilde \Phi_f \circ \Psi\).
  In particular, the generating invariant cones \(W_f := \Phi_f^{-1}(\hsp(V, f \circ \beta)_+)\) and \(\tilde W_f := \tilde \Phi_f^{-1}(\hsp(\tilde V, f \circ \tilde \beta)_+)\) are related by \(\tilde W_f = \Psi(W_f)\).
\end{rem}

\begin{rem}
  \label{rem:spindler-assumption-center}
  Let \(\g := \g(\fl,V,\fz,\beta)\) be admissible.
 
  (a) Let \(D \in \der(\g)\) be a diagonalizable derivation with \(D(\fl) \subset \fl\) and \(D(V) \subset V\) (cf.\ Theorem \ref{thm:adm-deriv-3grad-reducompl}). Then \(\fz = \fz(\g)\) and \(\fz_1 := [V,V] \subset \fz\) are also \(D\)-invariant. Since \(D\) is diagonalizable on these subspaces, we can find a \(D\)-invariant complement \(\fz_2 \subset \fz\) of \(\fz_1\). Then \(\g = \g(\fl,V,\fz_1,\beta) \oplus \fz_2\), where \(\beta\) is corestricted to \(\fz_1\), is a direct sum of Lie algebras and \(\g(\fl,V,\fz_1,\beta)\) is admissible. Hence, it is reasonable to only consider those admissible Lie algebras \(\g(\fl,V,\fz,\beta)\) with \([V,V] = \fz\) for our classification problem. It is clear from Spindler's construction that this is equivalent to \(\fz(\g) \subset [\g,\g]\).

  (b) Another assumption we usually make is that \(\fl\) acts \emph{effectively} on \(V\), that is, the ideal
  \[\fl_{\fix, V} := \{X \in \fl : X.V = \{0\}\}\]
  of \(\fl\) vanishes. If this is not the case, then there exists a complementary ideal \(\fl_{\eff, V}\) of \(\fl_{\fix, V}\) in \(\fl\) because \(\fl\) is reductive, and \(\g\) decomposes into the direct sum \(\g = \g(\fl_{\eff, V}, V, \fz, \beta) \oplus \fl_{\fix, V}\). Moreover, \(\fl_{\eff, V}\) acts effectively on \(V\) and \(\fl_{\fix, V}\) is a reductive admissible Lie algebra.
\end{rem}

\begin{lem}
  \label{lem:symp-mod-h1-homom}
  Let \(\g\) be a hermitian simple Lie algebra and let \((V,\Omega)\) be a symplectic \(\g\)-module of convex type. Let \(\rho: \g \rightarrow \sp(V,\Omega)\) be the homomorphism corresponding to the representation of \(\g\). Then the following assertions hold:
  \begin{enumerate}
    \item There exist \(H\)-elements in \(\g\) and \(\sp(V,\Omega)\) such that \(\rho\) is an \((H_1)\)-homomorphism.
    \item Suppose that \(\g\) is of tube type. Let \(h \in \g\) such that \(\g = \g_{-1}(h) \oplus \g_0(h) \oplus \g_1(h)\) and set \(h' := \rho(h)\). Then
      \[\sp(V,\Omega) = \sp(V,\Omega)_{-1}(h') \oplus \sp(V,\Omega)_0(h') \oplus \sp(V,\Omega)_1(h').\]
      In particular, \(\tau := 2\rho(h)\) is an antisymplectic involution on \(V\).
  \end{enumerate}
\end{lem}
\begin{proof}
  (a) is proven in \cite[Thm.\ IV.6]{Ne94a}.

  (b) We assume now that \(\g\) is of tube type.
  Using (a), we conclude that there exist \(H\)-elements \(H_0 \in \g, H_0' \in \sp(V,\Omega)\) such that \(\rho\) is an \((H_1)\)-homomorphism with respect to these \(H\)-elements.
  We obtain a Cartan decomposition \(\g = \fk \oplus \fp\) by setting \(\fk := \ker(\ad H_0)\) and \(\fp := [H_0,\g]\).
  Since all hyperbolic elements in \(\g\) are conjugate to elements in \(\fp\), we only have to prove (b) for those \(h\) contained in \(\fp\).
  Therefore, let \(h \in \fp\) and suppose that \(\g = \g_{-1}(h) \oplus \g_0(h) \oplus \g_1(h)\).
  By Lemma \ref{lem:herm-lie-3grad-h2-homo}, there exists an \((H_2)\)-homomorphism \(\kappa: (\fsl(2,\R), \frac{1}{2}U) \rightarrow (\g, H_0)\) with \(\kappa(\frac{1}{2}H) = h\).

  Hence, \(\varphi := \rho \circ \kappa : \fsl(2,\R) \rightarrow \sp(V,\Omega)\) is an \((H_1)\)-homomorphism.
  We claim that \(\varphi\) is actually an \((H_2)\)-homomorphism: By \cite[Prop.\ VIII.2.6]{Ne00}, there exists \(X \in \fz(\fk) = \R H_0\) such that \(\Omega(\rho(X)v,v) > 0\) for all \(v \in V \setminus \{0\}\).
  In particular, either \(\kappa(U) \in \fz(\fk)\) or \(-U \in \fz(\fk)\) satisfies this positivity condition, so that \((V,\Omega)\) is a symplectic \(\fsl(2,\R)\)-module of convex type via \(\varphi\).
  The discussion in Example \ref{ex:sympl-mod-conv-sl2} now shows that \(\varphi\) is an \((H_2)\)-homomorphism with respect to \(\frac{1}{2}U\) and \(\frac{1}{2}\varphi(U)\).

  By \cite[Ch.\ III \S1, Cor.\ 1.6]{Sa80}, this implies that \(\varphi(\frac{1}{2}H) = \rho(h) = h'\) induces a 3-grading on \(\g\).
  By identifying \(\sp(V,\Omega)\) with the standard symplectic Lie algebra \(\sp(2n,\R)\), where \(2n = \dim V\), we conclude with Example \ref{ex:herm-sp} that \(\tau = 2h' = 2\rho(h)\) is an antisymplectic involution.
\end{proof}

\subsection{The solvable case}
\label{sec:adm-wedgeclass-solvable}

According to Theorem \ref{thm:spindler-adm}, all solvable Lie algebras \(\g\) containing a pointed generating invariant closed convex cone are of the form \(\g = \g(\fa, V, \fz, \beta)\) in terms of Spindler's construction, where the reductive part \(\fa\) is an abelian Lie algebra. Moreover, \(\fa\) acts faithfully on the symplectic \(\fa\)-module \(V\), so that we may assume that \(\fa \subset \sp(V,\Omega)\) for some symplectic form \(\Omega = f \circ \beta\), where \(f \in \fz^*\).

Before stating the main theorem of this section (Theorem \ref{thm:adm-wedgeclass-solvable}), we discuss the derivations of the (generalized) Heisenberg algebra.

\begin{example}
  \label{ex:heisenberg-liealg-deriv}
  Let \((V,\Omega)\) be a symplectic vector space. The derivations of the Heisenberg algebra \(\g := \heis(V,\Omega) = V \oplus \R\) can be described as follows: Let \(D \in \der(\g)\). Since the center \(\fz(\g)\) of \(\g\) is \(D\)-invariant, we have \(D_\fz := D\lvert_{\fz(\g)} = c\id_\R\) for some \(c \in \R\). The restriction \(D_\fz\) extends to a derivation \(D_c \in \der(\g)\) via
  \[D_c(v,z) := (\tfrac{c}{2} v, cz), \quad v \in V, z \in \fz(\g).\]
  Let \(D_V := D - D_c \in \der(\g)\). Then
  \[D_V(\fz(\g)) = \{0\} \quad \text{and} \quad D_V(v,z) = (D_{V,V}v, D_{V,\fz}v) \quad \text{for } v \in V, z \in \fz(\g),\]
  where \(D_{V,V} \in \sp(V,\Omega)\) and \(D_{V,\fz} = \Omega(w,\cdot) = \ad(w)\) for some \(w \in V\).
\end{example}

\begin{definition}
  \label{def:gen-heis-alg}
  Let \(V,\fz\) be real vector spaces and let \(\beta: V \times V \rightarrow \fz\) be a skew-symmetric non-degenerate bilinear map.

  (a) The Lie algebra \(\heis(V,\beta) := V \oplus \fz\) with the bracket
  \[[(v,z), (v', z')] := (0, \beta(v,v'))\]
  is called a \emph{generalized Heisenberg algebra}.

  (b) Let \(D_V \in \End(V)\) and \(D_\fz \in \End(\fz)\). We say that the pair \((D_V,D_\fz)\) is \emph{\(\beta\)-compatible} if
  \begin{equation}
    \label{eq:deriv-spindler-perfect-Dz}
    D_\fz \beta(v,w) = \beta(D_V v, w) + \beta(v, D_V w) \quad \text{for all } v,w \in V.
  \end{equation}
\end{definition}

\begin{prop}
  \label{prop:deriv-gen-heisenberg}
  Let \(V\) and \(\fz\) be real vector spaces and let \(\beta: V \times V \rightarrow \fz\) be a skew-symmetric non-degenerate bilinear map. Let \(\g := \heis(V,\beta)\) be the corresponding generalized Heisenberg algebra and let \(D\) be a derivation of \(\g\). Then there exist linear maps \(D_V \in \End(V), D_{V,\fz} \in \Hom_\R(V,\fz)\), and \(D_\fz \in \End(\fz)\) such that the following assertions hold:
  \begin{enumerate}
    \item \(D(v,z) = (D_V v, D_{V,\fz}(v) + D_\fz z)\) for all \(v \in V, z \in \fz\).
    \item For \(D_\fz := D\lvert_\fz\), the pair \((D_V,D_\fz)\) is \(\beta\)-compatible. In particular, the map \((v,z) \mapsto (D_V v, D_\fz z)\) defines a derivation of \(\g\).
    \item The adjoint representation of \(D_V\) on \(\End(V)\) leaves the Lie subalgebra \(\sp(V,\beta)\) invariant.
  \end{enumerate}
\end{prop}
\begin{proof}
  Let \(D \in \der(\g)\). Then \(D\) leaves \(\fz = \fz(\g)\) invariant, so that \(D\) can be decomposed as in (a). Since \(D\) is a derivation, we have
  \[D_\fz \beta(v,w) = D[v,w] = [Dv, w] + [v,Dw] = \beta(D_V v, w) + \beta(v, D_V w)\quad \text{for all } v,w \in V,\]
  which proves (b).

  Let \(x \in \sp(V,\beta)\). Then we have
  \begin{align*}
    \beta([D_V, x]v,w) &= \beta(D_V x v, w) - \beta(xD_V v, w) = D_\fz \beta(xv,w) - \beta(xv, D_V w) + \beta(D_V v, xw) \\
                       &= D_\fz\beta(xv,w) + \beta(v,xD_Vw) + D_\fz \beta(v, xw) - \beta(v,D_V x w) \\
                       &= \beta(v,xD_V w) - \beta(v,D_V xw) = -\beta(v,[D_V,x]w)
  \end{align*}
  for all \(v,w \in V\) and therefore \([D_V,x] \in \sp(V,\beta)\). This proves (c).
\end{proof}

\begin{lem}
  \label{lem:adm-cone-int-u}
  Let \(\g\) be an admissible Lie algebra, let \(\fu \subset \g\) be the maximal nilpotent ideal of \(\g\), and let \(W \subset \g\) be a generating invariant closed convex cone with \(H(W) \subset \fz(\g)\). Then \(W \cap \fu \subset \fz(\g)\).
\end{lem}
\begin{proof}
  We first consider the case where \(H(W) = \{0\}\), i.e.\ \(W\) is pointed. Let \(\fa := \spann(W \cap \fu)\). Then \(\fa\) is a nilpotent ideal of \(\fu\) which is generated by the pointed invariant cone \(W \cap \fu\). Therefore, it is abelian by \cite[Ex.\ VII.3.21]{Ne00}. Moreover, \(\fa\) is an ideal of \(\g\) because \(\fu\) is an ideal of \(\g\) and \(W\) is \(\Inn(\g)\)-invariant. Thus, we have \(\fa \subset \fz(\g)\) by \cite[Prop.\ VII.2.23]{Ne00}.

  For a general generating invariant closed convex cone \(W \subset \g\) with \(H(W) \subset \fz(\g)\), we can apply the Factorization Theorem for Invariant Cones (cf.\ \cite[Prop.\ VIII.1.38]{Ne00}) to decompose \(W\) into \(W = W_1 + H(W)\), where \(W_1 \subset \g\) is a pointed generating invariant closed convex cone. Now \(H(W) \subset \fz(\g) \subset \fu\) and the previous argument imply that \(W \cap \fu = (W_1 \cap \fu) + (H(W) \cap \fu) \subset \fz(\g)\).
\end{proof}

The following theorem addresses the solvable case of our classification problem stated in the beginning of Section \ref{sec:liewedge-adm}: It says that, for a solvable admissible Lie algebra \(\g\) and a cone \(W_f\) specified by Spindler's construction, there exists no non-trivial derivation \(D\) of \(\g\) that induces a 3-grading on \(\g\) such that the \((\pm 1)\)-eigenspaces of \(D\) are generated by their intersection with \(W_f\).
Note that the theorem below even shows that this holds for the more general class of cones \(W \subset \g\) with \(H(W) \subset \fz(\g)\) and pointed cones in particular. For the cones \(W_f\), we have \(H(W_f) = \ker f \subset \fz(\g)\).

\begin{thm}
  \label{thm:adm-wedgeclass-solvable}
  Let \(\g\) be a solvable admissible Lie algebra with \(\fz(\g) \subset [\g,\g]\) and let \(W \subset \g\) be a generating invariant closed convex cone with \(H(W) \subset \fz(\g)\). Then there exists no non-zero derivation \(D \in \der(\g)\) such that \(\g = \g_{-1}(D) \oplus \g_0(D) \oplus \g_1(D)\) and \(\spann(\g_{\pm 1}(D) \cap W) = \g_{\pm 1}(D)\).
\end{thm}
\begin{proof}
  Suppose that \(D \in \der(\g)\) is a derivation with \(\g = \g_{-1}(D) \oplus \g_0(D) \oplus \g_1(D)\) and that we have \(\spann(\g_{\pm 1}(D) \cap W) = \g_{\pm 1}(D)\).
  By Theorem \ref{thm:adm-deriv-3grad-reducompl}, there exists a reductive subalgebra \(\fa\) of \(\g\) such that \(\g = \fu \rtimes \fa\) and \(D(\fz(\fa)) = \{0\}\).
  The subalgebra \(\fa\) is abelian because \(\g\) is solvable, so that \(D(\fa) = \{0\}\).
  By Lemma \ref{lem:adm-semprod-solv}, the subalgebra \(\fa\) is compactly embedded in \(\g\), and by Corollary \ref{cor:adm-semprod-spindler}, we have \(\fu = \fz(\g) \rtimes [\fa, \fu]\).
  Since \(\g_{\pm 1}(D) \subset \fu\) is generated by \(\g_{\pm 1} \cap W \subset \fu \cap W\) by assumption, Lemma \ref{lem:adm-cone-int-u} implies \(\g_{\pm 1}(D) \subset \fz(\g)\).
  Hence, we have \([\fa,\fu] \subset \g_0(D)\), which leads to
  \[D\fz(\g) = D[[\fa,\fu], [\fa,\fu]] = \{0\}\]
  and therefore \(D = 0\).
\end{proof}

\subsection{The general case}
\label{sec:deriv-adm-gen-case}

Since the Jacobi algebra \(\hsp(V,\Omega) = \heis(V,\Omega) \rtimes \sp(V,\Omega)\) of a symplectic vector space \((V,\Omega)\) (cf.\ Example \ref{ex:jacobi-alg}) is a prototypical example of a non-solvable, non-reductive Lie algebra, we study this example first. Lemma \ref{lem:der-levi-invcomp} shows that it suffices to consider those derivations \(D \in \der(\hsp(V,\Omega))\) that leave \(\sp(V,\Omega)\) invariant because all Levi complements are conjugate under inner automorphisms of \([\hsp(V,\Omega),\rad(\hsp(V,\Omega))]\) (cf.\ \cite[Thm.\ 5.6.13]{HN12}).

The simplest non-trivial example of an outer derivation of \(\hsp(V,\Omega)\) is given by
 \[D(v,z,x) := (v,2z,0) \quad \text{for } (v,z,x) \in V \times \R \times \sp(V,\Omega)\]
 (cf.\ Example \ref{ex:heisenberg-liealg-deriv}). We call the extension of \(\hsp(V,\Omega)\) by this derivation the \emph{conformal Jacobi algebra} \(\hcsp(V,\Omega) := \hsp(V,\Omega) \rtimes \R \id_V\). The following lemma shows that all derivations of the Jacobi algebra are restrictions of inner derivations of \(\hcsp(V,\Omega)\) to \(\hsp(V,\Omega)\).

For a subalgebra \(\fh \subset \g\) of a Lie algebra \(\g\), we define
\[\der(\g,\fh) := \{D \in \der(\g) : D(\fh) \subset \fh\}.\]

\begin{prop}
  \label{prop:jacobi-deriv}
  Every derivation of the Jacobi algebra \(\g := \hsp(V,\Omega) = \heis(V,\Omega) \rtimes \sp(V,\Omega)\) is induced by an element of the conformal Jacobi algebra \(\hcsp(V,\Omega)\). More precisely, for each \(D \in \der(\hsp(V,\Omega),\sp(V,\Omega))\) there exists \(c \in \R\) and \(h \in \sp(V,\Omega)\) such that
  \[D(v,z,x) = (cv + hv, 2cz, \ad(h)(x)), \quad (v,z,x) \in \hsp(V,\Omega).\]
\end{prop}
\begin{proof}
  Let \(\fs := \sp(V,\Omega)\) and \(D \in \der(\g, \fs)\). Since \(\fs \subset \g\) is simple and \(D\)-invariant, there exists \(h \in \fs\) such that \(D\lvert_{\fs} = \ad_{\fs}(h)\). We define \(D_h := \ad_{\g}(h)\). We have \(D\lvert_{\fz(\g)} = 2c \id_{\fz(\g)}\) for some \(c \in \R\). Denote by \(D_c \in \der(\g)\) the extension of \(D\lvert_{\fz(\g)}\) to \(\g\) by
  \[D_c(v,z,x) := (cv, 2cz, 0) \quad \text{for } (v,z,x) \in \g.\] 
  Since \(\fs\) and the maximal nilpotent ideal \(\fu := \heis(V,\Omega)\) are \(D\)-invariant, the space \(V = [\fs, \fu]\) is also \(D\)-invariant. Let \(\tilde D := D - D_c - D_h\). Then \(\tilde D(\fz(\g)) = \tilde D(\fs) = \{0\}\) and by Example \ref{ex:heisenberg-liealg-deriv} we have
  \[\tilde D(v,0,0) = (D - D_c - D_h)(v,0,0) = (yv, 0, 0) \quad \text{for } v \in V\]
  and some \(y \in \sp(V,\Omega)\). Hence, we have for all \(x \in \sp(V,\Omega)\) and \(v \in V\) that
  \[(yxv, 0, 0) = \tilde D(xv,0,0) = \tilde D[(0,0,x), (v, 0, 0)] = [(0,0,x), (yv, 0, 0)] = (xyv,0,0),\]
  which shows that \(y \in \fz(\fs) = \{0\}\). This implies that \(D = D_h + D_c\). 
\end{proof}

\begin{cor}
  \label{cor:jacobi-deriv-3grad}
  The derivations \(0 \neq D \in \der(\hsp(V,\Omega), \sp(V,\Omega))\) satisfying \(\g := \hsp(V,\Omega) = \g_{-1}(D) \oplus \g_0(D) \oplus \g_1(D)\) are exactly the linear maps of the form
  \begin{equation}
  \label{eq:jacobi-deriv-3grad}
    D(v,z,x) := (cv + \tfrac{1}{2}\tau_V v, 2cz, \tfrac{1}{2}[\tau_V, x]), \quad v \in V, z \in \R, x \in \sp(V,\Omega),
  \end{equation}
  where \(\tau_V\) is an antisymplectic involution on \(V\) and \(c \in \{\pm \frac{1}{2}\}\).
\end{cor}
\begin{proof}
  Let \(D \in \der(\hsp(V,\Omega), \sp(V,\Omega)) \setminus \{0\}\). By Proposition \ref{prop:jacobi-deriv}, \(D\) is of the form
  \[D(v,z,x) = (cv + hv, 2cz, \ad(h)(x)), \quad (v,z,x) \in \hsp(V,\Omega),\]
  for some \(c \in \R\) and \(h \in \sp(V,\Omega)\).
  Suppose that \(\g = \g_{-1}(D) \oplus \g_0(D) \oplus \g_1(D)\).
  Then \(c \in \{0, \pm\frac{1}{2}\}\) and \(\ad(h)\) induces a 3-grading on \(\fs := \sp(V,\Omega)\), so that we have \(h \neq 0\).
  By identifying \(\sp(V,\Omega)\) with the standard symplectic Lie algebra \(\sp(2n,\R)\), where \(2n = \dim(V)\), we see that this implies that \(\tau_V := 2h\) is an antisymplectic involution on \((V,\Omega)\) (cf.\ Example \ref{ex:herm-sp}). In particular, we have \(c \neq 0\).

  On the other hand, if \(D\) is of the form \eqref{eq:jacobi-deriv-3grad} and \(\tau_V\) is an antisymplectic involution on \((V,\Omega)\), then, if we consider \(\tau_V\) as an element of \(\sp(V,\Omega)\), we see that \(h:= \frac{1}{2}\tau_V\) induces a 3-grading on \(\sp(V,\Omega)\) with \(\spec(\ad(h)) = \{0, \pm 1\}\). Combining this with the assumption that \(c \in \{\pm \frac{1}{2}\}\), we see that \(D\) induces a 3-grading on \(\g\).
\end{proof}

\begin{rem}
  Let \(\g = \hsp(V,\Omega)\).
  For any non-zero derivation \(D \in \der(\hsp(V,\Omega), \sp(V,\Omega))\) with \(\g = \g_{-1}(D) \oplus \g_0(D) \oplus \g_1(D)\), the cones \(C_\pm = \g_{\pm 1}(D) \cap \hsp(V,\Omega)_+\) are generating in \(\g_{\pm 1}(D)\).
  This follows from Corollary \ref{cor:jacobi-deriv-3grad} and the arguments in \cite[Example 3.7]{Ne20}: Recall from Example \ref{ex:jacobi-alg} that \(\hsp(V,\Omega)\) can be identified with the space \(\Pol_{\leq 2}(V)\) of polynomials on \(V\) with degree at most two, so that \(\hsp(V,\Omega)_+\) corresponds to the subset of non-negative polynomials.
  By Corollary \ref{cor:jacobi-deriv-3grad}, the derivation \(D\) is of the form \eqref{eq:jacobi-deriv-3grad}.
  Let \(V_{\pm 1}\) be the \((\pm 1)\)-eigenspace of the involution \(\tau_V\) and suppose that \(c = \frac{1}{2}\).
  Then one can show that \(\g_{-1}(D) \cong \sp(V,\Omega)_{-1}(\tau_V)\) can be identified with the space \(\Pol_2(V_1)\) of polynomials of degree two on \(V_1\) and \(\g_1(D) \cong V_1 \oplus \fz(\g) \oplus \sp(V,\Omega)_1(\tau_V)\) can be identified with the space \(\Pol_{\leq 2}(V_{-1})\).
  In particular, the intersections \(\g_{\pm 1}(D) \cap \hsp(V,\Omega)_+\) are generating.
\end{rem}

\begin{lem}
  \label{lem:deriv-spindler-perfect}
  Let \(\g = \g(\fs,V,\fz,\beta)\) be a perfect admissible Lie algebra (cf.\ {\rm Remark \ref{rem:spindler-adm-levi-comm}})  and let \(D \in \der(\g)\) be a derivation of \(\g\) annihilating \(\fs\). Then there exists \(D_V \in \End_\fs(V)\) and \(D_\fz \in \End(\fz)\) such that
  \[D(v,z,x) = D_V v + D_\fz z \quad \text{for } v \in V, z \in \fz, x \in \fs.\]
  Moreover, the pair \((D_V,D_\fz)\) is \(\beta\)-compatible (cf.\ {\rm Definition \ref{def:gen-heis-alg}}).
  Conversely, every \(\beta\)-compatible pair \((D_V, D_\fz)\) with \(D_V \in \End_\fs(V)\) defines via \(D := D_V \oplus D_\fz\) a derivation on \(\g\) that annihilates \(\fs\).
\end{lem}
\begin{proof}
  Let \(D \in \der(\g)\) with \(D(\fs) = \{0\}\).
  Recall from Theorem \ref{thm:spindlers-constr} that, if \(\ft_\fs\) is a compactly embedded Cartan subalgebra of \(\fs\), then \(\ft := \{0\} \times \fz \times \ft_\fs\) is a compactly embedded Cartan subalgebra of \(\g\) and that \(V = [\ft,\fu]\), where \(\fu = V + \fz(\g)\) is the maximal nilpotent ideal of \(\g\).
  Since \(D\) commutes with \(\ad \fs\), it commutes in particular with \(\ad \ft = \ad \ft_\fs\), so that \(V\) is \(D\)-invariant.
  Moreover, \(\fz(\g) = \fz\) is \(D\)-invariant, so that there exists \(D_V \in \End(V)\) and \(D_\fz \in \End(\fz)\) such that \(D = D_V \oplus D_\fz\).
  The condition \([D,\ad(\fs)] = \{0\}\) implies that \(D_V \in \End_\fs(V)\) and \eqref{eq:deriv-spindler-perfect-Dz} holds because \(D\) is a derivation.

  Suppose conversely that the pair \((D_V,D_\fz)\) with \(D_V \in \End_\fs(V)\) and \(D_\fz \in \End(\fz)\) satisfies \eqref{eq:deriv-spindler-perfect-Dz}. It remains to show that \(D := D_V \oplus D_\fz \in \der(\g)\), which is equivalent to
  \[[D,\ad(x)] = \ad(Dx) \quad \text{for all } x \in \g.\]
  For \(x \in \fs\), this follows from \([D,\ad(\fs)] = \{0\}\) and \(D(\fs) = \{0\}\), and for \(x \in \fz\) from \(D(\fz) \subset \fz\). Let \(x \in V\) and \(y = (y_V,y_\fz,y_\fs) \in V \times \fz \times \fs = \g\). Then
  \[[D,\ad(x)]y = D_\fz[x,y_V] + D_V[x,y_\fs] - [x,D_Vy_V] = [D_Vx,y_V] + [D_Vx,y_\fs] = \ad(Dx)y,\]
  which finishes the proof.
\end{proof}

\begin{thm} {\rm (Classification Theorem)}
  \label{thm:adm-deriv-3grad}
  Let \(\g\) be an admissible Lie algebra with \(\fz(\g) \subset [\g,\g]\) and let \(D \in \der(\g)\). Then \(\g = \g_{-1}(D) \oplus \g_0(D) \oplus \g_1(D)\) if and only if \(\g\) is of the form \(\g = \g(\fl,V,\fz,\beta)\) in terms of Spindler's construction, where \(V,\fz,\fl,\) and \(\beta\) have the following properties:
  \begin{enumerate}[label={\rm (\arabic*)}]
    \item \(D(\fz(\fl)) = \{0\}\).
    \item Let \(\fs := [\fl,\fl]\). Then \(D\lvert_\fs = \ad(h)\) for some \(h \in \fs\) with \(\fs = \fs_{-1}(h) \oplus \fs_0(h) \oplus\fs_1(h)\). 
    \item Let \(\rho: \fl \rightarrow \sp(V,\beta)\) be the homomorphism corresponding to the representation of \(\fl\) on \(V\). Then \(D\lvert_V = D_V + \rho(h)\), where \(D_V \in \End_\fl(V)\) is diagonalizable on \(V\) with \(\spec(D_V) \subset \{0, \pm \frac{1}{2}\}\). Moreover, \(\ker(D_V) = \ker(\rho(h))\) and \(V_{-1/2}(D_V) \oplus V_{1/2}(D_V) = V_{-1/2}(\rho(h)) \oplus V_{1/2}(\rho(h))\).
  \end{enumerate}
  If these conditions are satisfied, then the following holds:
  \begin{enumerate}
    \item Suppose that \(\fl = \fs_0 \oplus \bigoplus_{k=1}^m \fs_k\), where \(\fs_0\) is compact and each \(\fs_k\) is hermitian simple for \(1 \leq k \leq m\). Let \(\fs_{\ell_1},\ldots,\fs_{\ell_n}\) be the hermitian simple ideals that act non-trivially on the submodule \(V_D := V_{-1/2}(D_V) \oplus V_{1/2}(D_V)\). Then each \(\fs_{\ell_k}\) is of tube type for \(1 \leq k \leq n\) and \((V_D, f \circ \beta)\) is a symplectic \((\fs_{\ell_1} \oplus \ldots \oplus \fs_{\ell_n})\)-module of convex type for every \(f \in \fz^*\) for which \((V, f \circ \beta)\) is a symplectic \(\fl\)-module of convex type.
    \item We have
      \begin{align*}
        \fz &= [V_{1/2}(D_V), V_{1/2}(D_V)] \oplus [V_{-1/2}(D_V), V_{-1/2}(D_V)]\\
            &\oplus ([V_0(D_V), V_0(D_V)] + [V_{-1/2}(D_V),V_{1/2}(D_V)]).
      \end{align*}
    \item For \(D_\fz := D\lvert_{\fz}\), we have \(D_\fz \beta(v,w) = \beta(D_V v, w) + \beta(v, D_V w)\) for all \(v,w \in V\), that is, the pair \((D_V, D_\fz)\) is \(\beta\)-compatible.
  \end{enumerate}
\end{thm}
\begin{proof}
  Without loss of generality, we may assume that \(\fl\) acts effectively on \(V\) (cf.\ Remark \ref{rem:spindler-assumption-center}(b)). We first note that the conditions (1)-(3) are sufficient for \(D\) to induce a 3-grading on \(\g\).

  Suppose conversely that \(D\) induces a 3-grading on \(\g\) as above.
  Let \(\fu\) be the maximal nilpotent ideal of \(\g\).
  By Theorem \ref{thm:adm-deriv-3grad-reducompl}, there exists a reductive \(D\)-invariant quasihermitian subalgebra \(\fl \subset \g\) such that \(\fu \rtimes \fl\) and \(D(\fz(\fl)) = \{0\}\).
  We fix a compactly embedded Cartan subalgebra \(\ft_\fs \subset \fs\), so that \(\ft := \fz(\g) + \fz(\fl) + \ft_\fs\) is a compactly embedded Cartan subalgebra of \(\g\) (cf.\ Theorem \ref{thm:adm-semprod}(b)).
  By Corollary \ref{cor:adm-semprod-spindler}, we now obtain \(\g = \g(\fl,V,\fz,\beta)\) in terms of Spindler's construction, where \(V := [\ft,\fu]\), \(\fz := \fz(\g)\), and \(\beta := [\cdot,\cdot]\lvert_{V \times V}\).

  The Lie algebra \(\fs = [\fl,\fl]\) is quasihermitian and semisimple, hence it can be written as a direct sum \(\fs = \fs_0 \oplus \bigoplus_{k=1}^m \fs_k\), where \(\fs_0\) is compact and \(\fs_k\) is a hermitian simple Lie algebra for \(k=1,\ldots,m\).
  We fix a linear functional \(f \in \fz^*\) such that the bilinear form \(\Omega := f \circ \beta\) is symplectic and \((V,\Omega)\) is a symplectic \(\fl\)-module of convex type (cf.\ Theorem \ref{thm:spindler-adm}).
  Moreover, we fix an \(\Omega\)-orthogonal decomposition into \(\fl\)-submodules
  \begin{equation}
    \label{eq:adm-deriv-3grad-decomp}
  (V,\Omega) = (V_0,\Omega_0) \oplus \bigoplus_{k=1}^m (V_k,\Omega_k)
  \end{equation}
  such that \(V_0\) is a submodule on which all hermitian simple ideals of \(\fs\) act trivially and \(V_k\) is a submodule for which \(\fs_k\) is the only hermitian simple ideal of \(\g\) in the above decomposition that acts non-trivially on it (cf.\ Theorem \ref{thm:sympmodconv-decomp}).
  By Lemma \ref{lem:sympmodconv-simple-herm-ideal}, each submodule \(V_k\) can be decomposed into orthogonal \(\fs_k\)-submodules \(V_{k,\eff} \oplus V_{k,\fix}\) such that \(\fs_k\) acts trivially on \(V_{k,\fix}\) and \(V_{k,\eff}\) is a symplectic \(\fs_k\)-module of convex type.

   Since \(D\) leaves \(\fs\) invariant, its restriction to this subalgebra is of the form \(\ad_\fs(h)\) for some \(h \in \fs\). By Lemma \ref{lem:deriv-spindler-perfect}, the derivation \(D\) is thus of the form
  \[D(v,z,x) = (D_V v + \rho(h)v, D_\fz z, [h,x]) \quad \text{for } v \in V, z \in \fz, x \in \fl,\]
  and some \(D_V \in \End_\fs(V), D_\fz \in \End(\fz)\).
  The pair \((D_V,D_\fz)\) is \(\beta\)-compatible, i.e.\ it satisfies \eqref{eq:deriv-spindler-perfect-Dz}. The fact that \(D(\fz(\fl)) = \{0\}\) implies that we even have \(D_V \in \End_\fl(V)\).

  It is now obvious from the description of \(D\) that the adjoint representation of \(h\) induces a 3-grading on \(\fs\).
  We can write \(h\) as \(h = \sum_{k=0}^m h_k\), where \(h_k \in \fs_k\), so that \(h_k\) induces a 3-grading on each ideal \(\fs_k\).
  Then we have \(h_k \neq 0\) only if \(\fs_k\) is a hermitian simple Lie algebra of tube type for \(k=0,\ldots,m\).
  Moreover, since \((V_{k,\eff},\Omega_k)\) is a symplectic \(\fs_k\)-module of convex type, the endomorphism \(\rho(h_k)\) is diagonalizable on \(V_{k,\eff}\) with eigenvalues \(\pm \frac{1}{2}\) if \(\fs_k\) is hermitian simple and of tube type and \(h_k \neq 0\), and it is trivial otherwise (cf.\ Lemma \ref{lem:symp-mod-h1-homom}).
  Since \(D_V\) commutes with \(\rho(h)\), it is therefore also diagonalizable.
  For each eigenvalue \(\lambda \in \R\) of \(D_V\), the eigenspace \(V_\lambda(D_V)\) is a symplectic \(\fl\)-module of convex type (cf.\ Lemma \ref{lem:sympmodconv-submodules}), so that \(\beta\) must be non-zero when restricted to this subspace. Combined with \eqref{eq:deriv-spindler-perfect-Dz}, this implies that
  \begin{equation}
    \label{eq:adm-deriv-3grad-z-eig}
    \{0\} \neq \beta(V_\lambda(D_V), V_\lambda(D_V)) \subset \fz_{2\lambda}(D_\fz).
  \end{equation}
  As a result, we have \(\spec(D_V) \subset \{0, \pm \frac{1}{2}\}\).
  This also proves
  \[\ker(D_V) = \ker(\rho(h)) \quad \text{and} \quad  V_{-1/2}(D_V) \oplus V_{1/2}(D_V) = V_{-1/2}(\rho(h)) \oplus V_{1/2}(\rho(h)).\]

  It remains to show that (a)-(c) hold under the above assumptions. (c) has already been shown.
  In order to prove (a) and (b), we recall that \(V_D\) and \(\ker(\rho(h))\) are adapted to the decomposition \eqref{eq:adm-deriv-3grad-decomp} in the sense that
  \[V_D = \bigoplus_{1 \leq k \leq m, h_k \neq 0} V_{k,\eff} \quad \text{and} \quad \ker(\rho(h)) = V_0 \oplus \bigoplus_{k=1}^m V_{k,\fix} \oplus \bigoplus_{1 \leq k \leq m, h_k = 0} V_{k,\eff}.\]
  If \(h_k \neq 0\) for \(0 \leq k \leq m\), then \(\fs_k\) is hermitian simple and of tube type.
  Since exactly one \(\fs_\ell\) acts non-trivially on \(V_k\) for \(1 \leq k, \ell, \leq m\), this shows that only those hermitian simple ideals of \(\fl\) which are of tube type can act non-trivially on \(V_D\).
  Hence, we have proven (a).
  Moreover, the above decomposition of \(V_D\) and \(\ker(\rho(h))\) shows that \(V_D\) and \(\ker(\rho(h))\) are \(\Omega\)-orthogonal.
  Since we can apply the same argument to every \(g \in \fz^*\) such that \((V,g \circ \beta)\) is a symplectic \(\fl\)-module of convex type and since the set of these functionals generates \(\fz^*\) (cf.\ Theorem \ref{thm:spindler-adm}), this implies that \(V_D\) and \(\ker(\rho(h))\) are also \(\beta\)-orthogonal.
  That the sum in (b) is indeed direct follows from \eqref{eq:adm-deriv-3grad-z-eig}.
\end{proof}

\begin{cor}
  Let \(\g\) be an admissible Lie algebra with \(\dim \fz(\g) = 1\) and let \(0 \neq D \in \der(\g)\) such that \(\g = \g_{-1}(D) \oplus \g_0(D) \oplus \g_1(D)\). Then \(\g\) is of the form \(\g = \g(\fl,V, \fz, \beta)\) in terms of Spindler's construction and \(D\) is of the form
  \[D(v,z,x) = (\tfrac{\lambda}{2}v + h.v, \lambda z, [h,x]) \quad \text{for } (v,z,x) \in V \times \fz \times \fl = \g,\]
  where \(\lambda \in \{\pm 1\}\) and \(h \in \fl \setminus \{0\}\) is such that \(\fl = \fl_{-1}(h) \oplus \fl_0(h) \oplus \fl_1(h)\).
\end{cor}
\begin{proof}
  That \(\g\) and \(D\) are of the form
  \[D(v,z,x) = (D_V v + h.v, \lambda z, [h,x]) \quad \text{for } (v,z,x) \in V \times \fz \times \fl = \g\]
  and some diagonalizable endomorphism \(D_V \in \End_\fl(V)\) with \(\spec(D_V) \subset \{0, \pm \frac{1}{2}\}\) and \(\lambda \in \{0, \pm 1\}\) follows from Theorem \ref{thm:adm-deriv-3grad}.
  Denote the representation of \(\fl\) on \(V\) by \(\rho: \fl \rightarrow \sp(V,\beta)\).

  Since \(\fz \cong \fz(\g)\) is one-dimensional, the direct sum in Theorem \ref{thm:adm-deriv-3grad}(b) only contains one summand.
  If \(\mu \in \R\) is an eigenvalue of \(D_V\), then \(V_\mu(D_V)\) is a symplectic \(\fl\)-module of convex type (cf.\ Lemma \ref{lem:sympmodconv-submodules}) so that
  \[\{0\} \neq \spann \beta(V_\mu(D_V), V_\mu(D_V)) = \fz_{2\mu}(D\lvert_\fz) = \fz.\]
  Thus, \(2\mu = \lambda\) and therefore \(D_V = \mu \id_V\). The case \(\mu = \lambda = 0\) would imply that \(h = 0\), which is not possible because \(D \neq 0 \) by assumption. Thus, \(\lambda \in \{\pm 1\}\).
\end{proof}

With the description of all derivations \(D\) of admissible Lie algebras \(\g\) that induce a 3-grading on \(\g\), we can now check whether \(\g_{\pm 1}(D) \cap W\) generates \(\g_{\pm 1}(D)\). Here, \(W\) is the invariant closed convex cone that we defined in the beginning of Section \ref{sec:liewedge-adm} via the natural homomorphism into the Jacobi algebra.

\begin{lem}
  \label{lem:adm-3grad-intpoints}
  Let \(\g = \g(\fs, V, \fz, \beta)\) be an admissible Lie algebra, where \(\fs\) is simple hermitian and of tube type. Let \(f \in \fz^*\) such that \(\Omega := f \circ \beta\) is symplectic and \((V,\Omega)\) is a symplectic \(\fs\)-module of convex type. Denote the corresponding representation of \(\fs\) by \(\rho: \fs \rightarrow \sp(V,\Omega)\). Moreover, let \(0 \neq h \in \fs\) such that \(\fs = \fs_{-1}(h) \oplus \fs_0(h) \oplus \fs_1(h)\). Then, for \(\lambda \in \{\pm 1\}\), there exists \(x \in \fs_{\lambda}(h) \cap W_f\) such that, for all \(z \in f^{-1}(\R_{> 0})\), the element \(x + z\) is an interior point of the cone
  \[W_f \cap \g_\lambda, \quad \text{where} \quad \g_\lambda := V_{\lambda/2}(\rho(h)) \oplus \fz \oplus \fs_{\lambda}(h).\]
\end{lem}
\begin{proof}
  Since \(\fs\) is simple, we may without loss of generality assume that \(\fs \subset \sp(V,\Omega)\).
  Suppose that \(\lambda = 1\). 
  By Lemma \ref{lem:symp-mod-h1-homom}, the adjoint representation of \(h\) on \(\sp(V,\Omega)\) also induces a 3-grading on this Lie algebra.
  We choose elements \(x \in \fs_1(h), y \in \fs_{-1}(h),\) and a Cartan involution \(\theta\) of \(\sp(V,\Omega)\) such that \(\fh := \la x,h,y\ra_{\mathrm{Lie}} \cong \fsl(2,\R)\) and such that the restriction \(\theta\lvert_{\fh}\) coincides with the standard Cartan involution \(X \mapsto -X^{\tran}\) on \(\fsl(2,\R)\) (cf.\ \cite[Ex.\ 2.29]{Oeh20}).  Using the Kantor--Koecher--Tits construction, we can thus endow \(\sp(V,\Omega)_1(h)\) with the structure of a simple euclidean Jordan algebra in such a way that \(x\) is a Jordan unit (cf.\ \cite[Ex.\ 2.25]{Oeh20}). In particular, there exists a symplectic basis of \(V\) such that \(h\) and \(x\) are represented by the matrices
  \[h = \frac{1}{2} \pmat{\bbone & 0 \\ 0 & -\bbone} \quad \text{and} \quad x = \pmat{0 & \bbone \\ 0 & 0}.\]

  Since \((V,\Omega)\) is a symplectic \(\fs\)-module of convex type, the intersection \(W_f \cap \fs\) is a non-trivial and thus generating invariant closed convex cone in \(\fs\). By \cite[Lem.\ 2.30]{Oeh20}, the intersection \(W_f \cap \fs_{1}(h)\) thus coincides up to sign with the cone of squares in the Jordan algebra \(\fs_1(h)\), which generates \(\fs_{1}(h)\). Similarly, we see that \(W_f \cap \fs_{-1}(h)\) generates \(\fs_{-1}(h)\).

  Let \(n := \frac{\dim V}{2}\). Recall from Example \ref{ex:jacobi-alg} that we can identify \(\hsp(\R^{2n})\) with the Lie algebra \(\Pol_{\leq 2}(\R^{2n})\) of polynomials of degree at most \(2\) and that \(\hsp(\R^{2n})_+\) can be identified with the cone of non-negative polynomials.
  Consequently, by using the above descriptions of \(h\) and \(x\), we can identify \(\g_1\) with the space of polynomials of degree at most \(2\) in \(n\) variables \(x_1,\ldots,x_n\).
  Then \(x\) corresponds to the polynomial \(p_x(x_1,\ldots,x_n) := \sum_{k=1}^n x_k^2\).
  Thus, for every \(z \in \fz\) such that \(f(z) > 0\), the element \(x + z\) corresponds to a strictly positive polynomial in \(n\) variables.
  Hence \(x + z\) is an interior point in \(W_f \cap \g_1\).
  The case \(\lambda = -1\) can be treated analogously.
\end{proof}

We now prove our main result: For every admissible Lie algebra \(\g\) with a 3-grading \linebreak\(\g = \g_{-1} \oplus \g_0 \oplus \g_1\) induced by a derivation and every generating invariant cone \(W_f\) induced by Spindler's construction, the intersection of \(W_f\) with \(\g_{\pm 1}\) generates \(\g_{\pm 1}\).

\begin{thm}{\rm (Non-degeneration Theorem)}
  \label{thm:adm-3grad-intcones}
  Let \(\g = \g(\fl, V, \fz, \beta)\) be an admissible Lie algebra and let \(D \in \der(\g)\) such that \(\g = \g_{-1}(D) \oplus \g_0(D) \oplus \g_1(D)\). Let \(f \in \fz^*\) be such that \((V,f \circ \beta)\) is a symplectic \(\fl\)-module of convex type. Then \(\spann(W_f \cap \g_{\pm 1}(D)) = \g_{\pm 1}(D)\).
\end{thm}
\begin{proof}
  We define \(\Omega := f \circ \beta\) and \(\fs := [\fl,\fl]\) as the semisimple part of \(\fl\).
  By Theorem \ref{thm:adm-deriv-3grad} and Remark \ref{rem:spindler-jacobi-hom-reducomp}, we may without loss of generality assume that \(D\) is of the form
  \[D(v,z,x) = (D_V v + \rho(h)v, D_\fz z, [h,x]), \quad \text{for } v \in V, z \in \fz, x \in \fl,\]
  and some \(D_V \in \End_\fl(V), D_\fz \in \End(\fz),\) and \(h \in \fs\).
  Moreover, \(\ad_\fs(h)\) induces a 3-grading on \(\fs\), the pair \((D_V,D_\fz)\) satisfies \eqref{eq:deriv-spindler-perfect-Dz}, and we have \(\ker(\rho(h)) = \ker(D_V)\).
  We may also assume that \(\fl\) acts effectively on \(V\) (cf.\ Remark \ref{rem:spindler-assumption-center}(b)).

  We decompose \(\fs\) into a direct sum \(\fs = \fs_0 \oplus \bigoplus_{k=1}^m \fs_k\), where \(\fs_0\) is compact and \(\fs_k\) is a hermitian simple ideal of \(\fs\) for \(1 \leq k \leq m\).
  Moreover, we decompose \((V,\Omega)\) into an \(\Omega\)-orthogonal direct sum \(V = V_0 \oplus \bigoplus_{k=1}^m V_k\) of submodules such that every hermitian simple ideal of \(\fs\) acts trivially on \(V_0\) and \(\fs_k\) acts on \(V_\ell\) non-trivially if and only if \(k = \ell\) for \(1 \leq k,\ell \leq m\) (cf.\ Theorem \ref{thm:sympmodconv-decomp}).
  Since \(\ker(\rho(h)) = \ker(D_V)\) and all ideals \(\fs_k\) acting non-trivially on \(V_{\pm 1}(D)\) for \(1 \leq k \leq m\) are hermitian simple and of tube type (cf.\ Theorem \ref{thm:adm-deriv-3grad}), it suffices to prove the claim in the case where all hermitian simple ideals of \(\fs\) are of tube type.
  Moreover, we may without loss of generality assume that \([h, \fs_k] \neq \{0\}\) and that each \(V_k\) is a symplectic \(\fs_k\)-module of convex type for \(1 \leq k \leq m\) (cf.\ Lemma \ref{lem:sympmodconv-simple-herm-ideal} and Theorem \ref{thm:adm-deriv-3grad}(a)).

  Let \(\beta_k := \beta\lvert_{V_k \times V_k}\) and write \(h = \sum_{k=1}^m h_k\), where \(h_k \in \fs_k\) for \(1 \leq k \leq m\). By applying Lemma \nolinebreak\ref{lem:adm-3grad-intpoints} to the admissible Lie algebras \(\g(\fs_k, V_k, \fz, \beta_k)\), we obtain elements \(x_k \in \fs_1(h) \cap \fs_k\) such that, for all \(z \in f^{-1}(\R_{> 0}) \cap \fz_1(D) \neq \{0\}\), the element \(x_k + z\) is an interior point of the cone
  \[W_f \cap (V_{1/2}(h_k) \oplus \fz \oplus (\fs_1(h) \cap \fs_k))\]
  Since the modules \(V_{1/2}(h_k)\) and \(V_{1/2}(h_\ell)\) are \(\Omega\)-orthogonal and \(\fs_k.V_\ell = \{0\}\) for \(k \neq \ell\), the element \(x := \sum_{k=1}^m x_k\) has the property that \(x + z\) is contained in the interior of the cone
  \[W_f \cap (V_{1/2}(h) \oplus \fz \oplus \fs_1(h)).\]
  Furthermore, as \(x + z \in \fz_1(D) \oplus \fs_1(h)\), it is in particular contained in the interior of the cone
  \[W_f \cap ((V_{1/2}(h) \cap V_{1/2}(D_V)) \oplus \fz_1(D) \oplus \fs_1(h)) = W_f \cap \g_1(D).\]
  In particular, \(W_f \cap \g_1(D)\) has non-empty interior, so that it generates \(\g_1(D)\). Analogously, one can show that \(W_f \cap \g_{-1}(D)\) generates \(\g_{-1}(D)\).
\end{proof}

\begin{rem}
  \label{rem:gen-sp-deriv}
  Theorem \ref{thm:adm-deriv-3grad} shows that, in order to determine the derivations of an admissible Lie algebra \(\g = \g(\fl,V,\fz,\beta)\) which induce 3-gradings, an important step is to determine those pairs \((D_V, D_\fz) \in \End(V) \times \End(\fz)\) which satisfy the equation
  \[\beta(D_V v, w) + \beta(v, D_V w) = D_\fz \beta(v,w) \quad \text{for all } v,w \in V.\]
  In case \(\beta\) is given explicitly, the following observation might be useful: Recall from Theorem \ref{thm:adm-gen-jacobi-alg} that the Lie algebra \(\sp(V,\beta)\) is reductive and admissible because \(\g\) is admissible.
  Let \(X \in \sp(V,\beta)\) and let \((D_V,D_\fz)\) be a pair as above.
  Then Proposition \ref{prop:deriv-gen-heisenberg} shows that \(\ad(D_V)(X) \in \sp(V,\beta)\).
  In particular, \(\ad(D_V)\) restricts to a derivation of the semisimple Lie algebra \([\sp(V,\beta), \sp(V,\beta)]\).
  Hence, there exists \(E_V \in \sp(V,\beta)\) and \(F_V \in \End_{[\sp(V,\beta), \sp(V,\beta)]}(V)\) such that \(D_V = E_V + F_V\).
  In particular, the pair \((F_V, D_\fz)\) satisfies the above equation. 
\end{rem}

\begin{example}
  \label{ex:adm-wedge-center-2dim}
  Let \((V,\Omega)\) be a symplectic vector space and let \(\fs := \sp(V,\Omega)\). We define on the symplectic vector space \(W := (V \oplus V, \Omega \oplus \Omega)\) a representation of \(\fs\) via
  \[\rho: \sp(V,\Omega) \rightarrow \sp(V \oplus V, \Omega \oplus \Omega), \quad \rho(x)(v_1,v_2) := (\rho(x)v_1,\rho(x)v_2).\]
  Moreover, we consider the \(\fs\)-invariant bilinear map
  \[\beta : (V \oplus V) \times (V \oplus V) \rightarrow \R^2 \quad q(v,w) := (\Omega(v_1,w_1),\Omega(v_2,w_2)).\]
  Setting \(f(\lambda,\mu) := \lambda + \mu\), where \(\lambda,\mu \in \R\), we obtain a symplectic \(\fs\)-module \((V \oplus V, f \circ \beta)\) of convex type, so that \(\g := \g(\fs,V \oplus V, \R^2,\beta)\) is a Lie algebra containing a pointed generating invariant closed convex cone by Theorem \ref{thm:spindler-adm}.

  Let \(h \in \fs\) be such that \(\fs = \fs_{-1}(h) \oplus \fs_0(h) \oplus \fs_1(h)\). Then \(\spec(\ad \rho(h)) = \{0, \pm 1\}\) and therefore \(\spec(\rho(h)) = \{\pm \frac{1}{2}\}\) (cf.\ Example \ref{ex:herm-sp}). Moreover, consider the linear maps
  \[D_V(v_1,v_2) := \tfrac{1}{2}(v_1,-v_2), \, v_1,v_2 \in V, \quad \text{and} \quad D_\fz(\lambda,\mu) := (\lambda,-\mu), \, \lambda,\mu \in \R.\]
  Then
  \[D(v,z,x) := (D_Vv + \rho(h)v, D_\fz z, [h,x]), \quad (v,z,x) \in V \times \R^2 \times \fs,\]
  defines a derivation on \(\g\) by Lemma \ref{lem:deriv-spindler-perfect}. The derivation \(D\) induces a 3-grading on \(\g\) with
  \[\g_0(D) = (V_{-1/2}(D_V) \cap V_{1/2}(\rho(h))) + (V_{1/2}(D_V) \cap V_{-1/2}(\rho(h))) \oplus \fs_0(h) \quad \text{and}\]
  \[\g_{\pm 1}(D) = (V_{\pm 1/2}(D_V) \cap V_{\pm 1/2}(\rho(h))) \oplus \R^2_{\pm 1}(D_\fz) \oplus \fs_{\pm 1}(h).\]
  The cone \(W_f \cap \g_{\pm 1}(D)\) generates \(\g_{\pm 1}(D)\) by Theorem \ref{thm:adm-3grad-intcones}.
\end{example}

\begin{example}
  Let \(\fs\) be a simple hermitian Lie algebra of tube type and let \((V,\Omega)\) be a symplectic \(\fs\)-module of convex type. We denote the corresponding representation of \(\fs\) by \(\rho: \fs \rightarrow \sp(V,\Omega)\). Then the exterior product \(V \wedge V\) becomes a semisimple \(\fs\)-module with the usual action of \(\fs\), so that
  \[V \wedge V = (V \wedge V)_{\fix} \oplus (V \wedge V)_{\eff}\]
  is a direct sum of \(\fs\)-modules. We denote by \(p : V \wedge V \rightarrow (V \wedge V)_{\fix}\) the projection onto \((V \wedge V)_{\fix}\) along this decomposition. Consider now the skew-symmetric bilinear map
  \[\beta: V \times V \rightarrow (V \wedge V)_{\fix}, \quad (v,w) \mapsto p(v \wedge w).\]
  By construction, we have \(\rho(\fs) \subset \sp(V,\beta)\). Hence, we can apply Spindler's construction to obtain a Lie algebra \(\g := \g(\fs, V, (V \wedge V)_{\fix}, \beta)\). If we interpret the skew-symmetric bilinear form \(\Omega\) as a linear functional on \(V \wedge V\), then we see that
  \[\Omega((V \wedge V)_{\eff}) = \Omega(\spann(\fs.(V \wedge V))) = \{0\},\]
  so that \((\Omega \circ \beta)(v,w) = \Omega(v,w)\) for all \(v,w \in V\). In particular, \(\g\) is admissible by Theorem \ref{thm:spindler-adm} because \((V,\Omega)\) is of convex type.

  (a) We study the derivations of \(\g\) that induce 3-gradings.

  We first note that the isomorphism
  \[\varphi_\Omega : (V \otimes V)_{\fix} \rightarrow \End_\fs(V), \quad \text{determined by} \quad \varphi_\Omega(v \otimes w)(a) := \Omega(v, a)w, \quad v,w,a \in V,\]
  maps the subspace \((V \wedge V)_{\fix}\) onto the set \(\End_\fs(V)^\#\) of \(\Omega\)-symmetric module endomorphisms, so that \((V \wedge V)_{\fix} \cong \End_\fs(V)^\#\). We assume from now on that \(V\) is isotypic because non-isomorphic submodules of \(V\) are \(\beta\)-orthogonal.

  Let \(V_1,V_2 \subset V\) be two non-trivial submodules of \(V\) with \(V_1 \cap V_2 = \{0\}\). Then \linebreak\(\beta(V_1,V_1) \cap \beta(V_2,V_2) = \{0\}\) because 
  \[\varphi_\Omega(\beta(V_k,V_k))(V) \subset V_k \quad \text{for } k=1,2.\]
  On the other hand, we see that \(\varphi_\Omega(\beta(V_k, V_\ell))\) is a set of \(\Omega\)-symmetric homomorphism from \(V_k\) to \(V_\ell\) when restricted to \(V_k\), where \(k,\ell \in \{1,2\}\).

  This observation implies that we have very few restrictions when constructing derivations on \(\g\) that induce 3-gradings. In fact, every derivation in \(\der(\g,\fs)\) with this property can be constructed as follows:
  Let \(h \in \fs\) such that \(\fs = \fs_{-1}(h) \oplus \fs_0(h) \oplus \fs_1(h)\) and let \(V = V_{-1} \oplus V_1\) be a decomposition of \(V\) into two submodules.  We define \(D_V \in \End_\fs(V)\) by \(D_V v := \frac{\lambda}{2} v\) for \(\lambda \in \{\pm 1\}, v \in V_\lambda\).
  Then the linear map \(D_\fz \in \End(\fz)\) defined by \(D_\fz v := (\lambda + \lambda') v\) for \(v \in \beta(V_\lambda, V_{\lambda'})\) and \(\lambda,\lambda' \in \{\pm 1\}\) is well defined because
  \[\fz = \bigoplus_{\lambda,\lambda' \in \{\pm 1\}} \beta(V_\lambda, V_{\lambda'}).\]
  Thus, the derivation defined by
  \[D(v,z,x) := (\rho(h)v + D_V v, D_\fz z, [h,x]) \quad \text{for } (v,z,x) \in \g\]
  induces a 3-grading on \(\g\). Conversely, every derivation of \(\g\) is of this form by Theorem \ref{thm:adm-deriv-3grad}.
  Notice that we did not assume that the decomposition \(V = V_{-1} \oplus V_1\) is \(\Omega\)-orthogonal in the argument above. In particular, this construction yields examples of derivations \(D\) where \(D_V\) is not symmetric with respect to \(\beta\).

  (b) Since \(\g\) is admissible, the generalized symplectic Lie algebra \(\sp(V,\beta)\) is reductive and admissible as well by Theorem \ref{thm:adm-gen-jacobi-alg}.
  The arguments made in the beginning of this example show that the linear functionals \(g \in \fz(\g)^*\) for which \((V, g \circ \beta)\) is a symplectic \(\fs\)-module of convex type naturally correspond to symplectic forms \(\Omega_g\) on \(V\) for which \((V,\Omega_g)\) is a symplectic \(\fs\)-module of convex type.
  Such a symplectic form is uniquely determined by an \(\Omega\)-symmetric module automorphism \(\Psi_g \in \Aut_\fs(V)^\#\) such that
  \[\Omega_g(v,w) = \Omega(\Psi_g v, w) \quad \text{for all } v,w \in V.\]
  Conversely, every \(\Omega\)-symmetric module automorphism \(\Psi \in \Aut_\fs(V)^\#\) specifies a symplectic form \(\Omega_\Psi\) such that \((V,\Omega_\Psi)\) is a symplectic \(\fs\)-module of convex type (cf.\ \cite[Lem.\ I.15]{Ne94a}). Combining this with the alternative characterization of \(\sp(V,\beta)\) from Remark \ref{rem:gen-symp-alg-char}, we obtain
  \[\sp(V,\beta) = \{x \in \sp(V,\Omega) : (\forall \Psi \in \Aut_\fs(V)^\#)\, [x,\Psi] = 0\}.\]
\end{example}

\appendix

\section{Algebraic groups and algebraic Lie subalgebras}
\label{sec:app-alg-grp}

\begin{definition}
  Let \(V\) be a real finite dimensional vector space and let \(\g \subset \gl(V)\).
  Then \(\g\) is called \emph{algebraic} if there exists an algebraic subgroup \(G \subset \GL(V)\) with \(\g = \L(G)\).
  We denote by \(a(\g)\) the smallest algebraic subalgebra of \(\gl(V)\) containing \(\g\).
  Note that the set of algebraic Lie subalgebras is stable under intersections.
\end{definition}

\begin{lem}
  \label{lem:subalg-algebraic-liebrack}
  Let \(V\) be a finite dimensional \(\K\)-vector space, where \(\K\) is a field of characteristic \(0\). Let \(\g,\fh \subset \gl(V)\) be Lie subalgebras. Then the following holds:
  \begin{enumerate}
    \item \([\fh,\fh] = [a(\fh),a(\fh)]\) is algebraic.
    \item If \([\fh,\g] \subset \g\), then \(a(\fh + \g) = a(\g) + a(\fh)\).
  \end{enumerate}
\end{lem}
\begin{proof}
  cf.\ \cite[Prop.\ I.6]{Ne94b}.
\end{proof}

\begin{lem}
  \label{lem:lie-adm-deriv-alg}
  Let \(\g\) be an admissible Lie algebra. Let \(D \in \der(\g)\) such that
  \[\g = \g_{-1}(D) \oplus \g_0(D) \oplus \g_1(D).\]
  Then the following holds:
  \begin{enumerate}
    \item \(a(\ad \g) = \ad \g + \ft_r\), where \(\ft_r \subset \der(\g)\) is an abelian Lie algebra.
    \item The Lie algebra \(\g_D := a(\ad(\g)) + \R D \subset \gl(\g)\) is algebraic.
  \end{enumerate}
\end{lem}
\begin{proof}
  (a) Since \(\g\) is admissible, it is of the form \(\g = \g(\fl,V,\fz,\beta)\) in terms of Spindler's construction (cf.\ Section \ref{sec:lie-adm}). In particular, the center \(\fz(\fl)\) of \(\fl\) is compactly embedded in \(\g\) (Theorem \ref{thm:spindlers-constr}).
  Hence, the group \(T_r := \oline{\la e^{\ad x} : x \in \fz(\fl)\ra_{\mathrm{grp}}} \subset \Aut(\g)\) is compact and therefore an algebraic group with \(\ad_\g \fz(\fl) \subset \L(T_r)\).

  Our assumptions on \(\g\) imply that \(\ad \g = \ad [\g,\g] + \ad \fz(\fl)\). Thus, we have
  \[a(\ad \g) = a([\ad \g,\ad \g]) + a(\ad_\g \fz(\fl)) = [\ad \g,\ad \g] + a(\ad_\g \fz(\fl)) = \ad \g + \L(T_r)\]
  (cf.\ Lemma \ref{lem:subalg-algebraic-liebrack} and Remark \ref{rem:spindler-adm-levi-comm}).
 
  (b) To see that \(\g_D\) is algebraic, we note that \([\R D, \ad \g] \subset \ad \g\), so that \([\R D, a(\ad \g)] \subset a(\ad \g)\) and therefore
  \[a(\g_D) = a(\ad \g) + a(\R D) = a(\ad \g) + \R D = \g_D.\qedhere\]
\end{proof}

\subsection*{Acknowledgement}
I am most grateful to Karl-Hermann Neeb for his support during my research towards this article and for reading earlier versions of the manuscript.

\end{document}